\documentclass[12pt]{amsart}
\usepackage{latexsym,amssymb,amsmath}
\usepackage{amsmath,amsfonts}
\usepackage{epsfig}
\usepackage{graphicx}
\usepackage{verbatim,xcolor}
\newtheorem{theorem}{Theorem}[section]
\newtheorem{lemma}[theorem]{Lemma}
\newtheorem{prop}[theorem]{Proposition}

\newtheorem{definition}[theorem]{Definition}
\newtheorem{rmk}[theorem]{Remark}

\allowdisplaybreaks

\newcommand{\Hmm}[1]{\leavevmode{\marginpar{\tiny%
$\hbox to 0mm{\hspace*{-0.5mm}$\leftarrow$\hss}%
\vcenter{\vrule depth 0.1mm height 0.1mm width \the\marginparwidth}%
\hbox to 0mm{\hss$\rightarrow$\hspace*{-0.5mm}}$\\\relax\raggedright #1}}}
\newcommand{\nc}{\newcommand}
\nc{\les}{\lesssim}
\nc{\ges}{\gtrsim}
\nc{\nit}{\noindent}
\nc{\nn}{\nonumber}
\nc{\D}{\partial}
\nc{\diff}[2]{\frac{d #1}{d #2}}
\nc{\diffn}[3]{\frac{d^{#3} #1}{d {#2}^{#3}}}
\nc{\pdiff}[2]{\frac{\partial #1}{\partial #2}}
\nc{\pdiffn}[3]{\frac{\partial^{#3} #1}{\partial{#2}^{#3}}}
\nc{\abs}[1] {\lvert #1 \rvert}
\nc{\cAc}{{\cal A}_c}
\nc{\cE}{{\cal E}}
\nc{\mF}{{\mathcal F}}
\nc{\cF}{{\mathcal F}} 
\nc{\cP}{{\cal P}}
\nc{\cV}{{\cal V}}
\nc{\cQ}{{\cal Q}}
\nc{\cGin}{{\cal G}_{\rm in}}
\nc{\cGout}{{\cal G}_{\rm out}}
\nc{\cO}{{\cal O}}
\nc{\Lav}{{\cal L}_{\rm av}}
\nc{\cL}{{\mathcal L}}
\nc{\cH}{{\mathcal H}}
\nc{\cB}{{\cal B}}
\nc{\cZ}{{\cal Z}}
\nc{\cR}{{\cal R}}
\nc{\cT}{{\cal T}}
\nc{\cY}{{\cal Y}}
\nc{\cX}{{\cal X}}
\nc{\cXT}{{{\cal X}(T)}}
\nc{\cBT}{{{\cal B}(T)}}
\nc{\vD}{{\vec \mathcal{D}}}
\nc{\efield}{\mathcal{E}}
\nc{\vE}{{\vec \efield}}
\nc{\vB}{{\vec \mathcal{B}}}
\nc{\vH}{{\vec \mathcal{H}}}
\nc{\ty}{{\tilde y}}
\nc{\tu}{{\tilde u}}
\nc{\tV}{{\tilde V}}
\nc{\Pc}{{\bf P_c}}
\nc{\bx}{{\bf x}}
\nc{\bX}{{\bf X}}
\nc{\bXYZ}{{\bf XYZ}}
\nc{\bY}{{\bf Y}}
\nc{\bF}{{\bf F}}
\nc{\bS}{{\bf S}}
\nc{\dV}{{\delta V}}
\nc{\dE}{{\delta E}}
\nc{\TT}{{\Theta}}
\nc{\dPsi}{{\delta\Psi}}
\nc{\order}{{\cal O}}
\nc{\Rout}{R_{\rm out}}
\nc{\eplus}{e_+}
\nc{\eminus}{e_-}
\nc{\epm}{e_\pm}
\nc{\eps}{\varepsilon}
\nc{\vnabla}{{\vec\nabla}}
\nc{\G}{\Gamma}
\nc{\w}{\omega}
\nc{\mh}{h}
\nc{\mg}{g}
\nc{\vphi}{\varphi}
\nc{\tlambda}{\tilde\lambda}
\nc{\be}{\begin{equation}}
\nc{\ee}{\end{equation}}
\nc{\ba}{\begin{eqnarray}}
\nc{\ea}{\end{eqnarray}}

\nc{\g}{\gamma}
\nc{\ol}{\overline}

\def\R{\mathbb R}

\nc{\T}{\mathbb T}
\nc{\Z}{\mathbb Z}
\nc{\N}{\mathbb N}
\nc{\pt}{\partial_t}
\nc{\la}{\langle}
\nc{\ra}{\rangle}
\nc{\infint}{\int_{-\infty}^{\infty}}
\nc{\halfwidth}{6.5cm}
\nc{\uu}{\" u}
\nc{\oo}{\" o}
\nc{\nlayers}{L} \nc{\nsectors}{M}
\nc{\indicator}{\mathbf{1}}
\nc{\Rhole}{R_{\rm hole}}
\nc{\Rring}{R_{\rm ring}}
\nc{\neff}{n_{\rm eff}}
\nc{\Frem}{F_{\rm rem}}
\nc{\DD}{\Delta}
\nc{\cD}{\mathcal D}
\nc{\lnorm}{\left\|}
\nc{\rnorm}{\right\|}
\nc{\rnormp}{\right\|_{\ell^{p,\eps}}}
\nc{\rar}{\rightarrow}
\nc{\sgn}{{\rm sign}}
\nc{\non}{\nonumber}
\nc{\wh}{\widehat}
\allowdisplaybreaks
\date{\today}

\begin{document}

\title[The Fifth Order KP--II Equation]{The Fifth Order KP--II Equation on the Upper Half--plane}

\author[ Erdo\u{g}an, G{\uu}rel, Tzirakis]{M. B. Erdo\u{g}an, T. B. G{\uu}rel, and N. Tzirakis}
\thanks{The first author is partially supported by NSF grant  DMS-1501041. The second author is partially supported by T\"UB\.ITAK grant 118F152 and Bo\u{g}azi\c{c}i University Research Fund grant BAP-14081. The third  author's work was supported by a grant from the Simons Foundation (\#355523 Nikolaos Tzirakis) and by Illinois Campus Research Board RB 18051}
 
\address{Department of Mathematics \\
University of Illinois \\
Urbana, IL 61801, U.S.A.}
\email{berdogan@illinois.edu}

\address{Department of Mathematics \\
Bo\u gazi\c ci University\\ 
Bebek 34342, Istanbul, Turkey}
\email{bgurel@boun.edu.tr}

\address{Department of Mathematics \\
University of Illinois \\
Urbana, IL 61801, U.S.A.}
\email{tzirakis@illinois.edu}

\begin{abstract}

In this paper we study the fifth order Kadomtsev--Petviashvili II (KP--II) equation on the upper half--plane $U=\{(x,y)\in \R^2: y>0\}$. In particular we obtain low regularity local well--posedness using the restricted norm method of Bourgain and   the Fourier--Laplace method of solving initial and boundary value problems.  Moreover we prove that the nonlinear part of the solution is in a smoother space than the initial data.

\end{abstract}

\maketitle

\section{Introduction}

In this paper we study the following initial-boundary value problem for the fifth order KP--II equation
\begin{equation}\label{kp2}
\left\{
\begin{array}{l}
\partial_x\big(u_{t}-\D^5_x u+uu_x\big) +u_{yy}=0 \quad x \in {\R}, \, y>0,\, t>0,\\
u(x,y,0)=g(x,y)\in H^{s}(U),\quad u(x,0,t)=h(x,t), 
\end{array}
\right.
\end{equation}
where $U=\{(x,y)\in\R^2: y> 0\}$ is the upper half--plane. For the boundary data $h$ the suitable space turns out to be an  $L^2$ based Sobolev space, $\cH^s_{x,t}$,  see \eqref{eq:cHs}. 
In addition, for  $s>\frac12$ we have  the   compatibility condition for the $L^2$ traces: $g|_{t=0}=h|_{t=0}$. The compatibility condition is necessary since the solutions we are interested in have  continuous $L^2_x$ traces for $s>\frac12$.  

Recently, there has been a lot of work dedicated to the fifth order KP--II equation when the domain is  $\R^2$, the two dimensional torus $\T^2$ or cylinders of the form $\R \times \T$. We refer the reader to the papers \cite{mst}, \cite{miao}, and the references therein. The  two dimensional model occurs naturally in the modeling of certain long dispersive waves. In \cite{kawa}, Kawahara derived the equation 
$$u_{t}+\partial_{x}^{5}u +\alpha \partial_{x}^{3}u+uu_x=0$$
to model solitary waves with an oscillatory structure  propagating in one direction, which cannot be
obtained from the classical KdV equation
$$u_{t}+\partial_{x}^{3}u+uu_x=0.$$
Taking into account weak transverse effects in the $y$ direction leads  to the  fifth order Kadomtsev--Petviashvili equation, \cite{kp}
$$\partial_x\big(u_{t}-\D^5_x u+\alpha \partial_{x}^{3}u+uu_x\big) +\sigma  u_{yy}=0,$$
where $\sigma =-1$ corresponds to the KP--I type equation while $\sigma= 1$ corresponds to the KP--II type equation. 
Thus the fifth order KP equation can be thought as the mixing of the KP  equation with a Kawahara term. Solitary waves for these equations were studied in \cite{kb,kb1}.

For the classical KP--I equation
$$\partial_x\big(u_{t}+ \partial_{x}^{3}u+uu_x\big) -u_{yy}=0$$ 
local and global and global well--posedness results are harder to obtain. This can be seen by using the dispersive symbol of the equation. In the case of KP--I, there is half derivative smoothing in the $x$ direction,  while for KP--II one can gain a full derivative. Since in our paper we concentrate on KP--II we refer the reader to \cite{mst} for recent results on KP--I. The low regularity well--posedness theory of the KP--II equation on the plane,
$$\partial_x\big(u_{t}+ \partial_{x}^{3}u+uu_x\big) +u_{yy}=0,$$  
started with the seminal paper of Bourgain in \cite{Bou1}. Bourgain obtained local well--posedness (LWP) and global well--posedness (GWP) for solutions with initial data $g\in L^{2}(\R^2)$. GWP follows easily from the locally well--posed solutions since the flow conserves the $L^2$ norm.
There are more recent results on anisotropic spaces of the form $H^{s_1}\times H^{s_2}$.  Note that the solution remains invariant under the  appropriate scaling of  the initial data if $s_1+2s_2=-\frac12$. In addition, the restriction $s_2 \geq 0$ is natural due to the Galilean invariance of the equation \cite{bm}. Takaoka and Tzvetkov, \cite{tt},  proved LWP for any $s_1>-\frac13$ and $s_2\geq 0$. Takaoka,  \cite{ht}, further improved this result going down to $s_1>-\frac12$ and $s_2\geq 0$. However, Takaoka's result  requires an additional  low frequency assumption. This assumption was later removed by Hadac in \cite{hadac}. Finally, the critical regularity ($s_1=-\frac12,\ s_2=0$)  was reached in \cite{hhk} with the additional assumption of small initial data.

For the fifth order KP--II the scaling relation is $s_1+3s_2=-2$. Saut and Tzvetkov in \cite{saut} proved LWP in anisotropic Sobolev spaces when $s_1>-\frac14$ and $s_2 \geq 0$. They also proved GWP for initial data in $L^2(\R^2)$. Their result was improved by Isaza et al. in \cite{isaza} where   LWP was established for $s_1>-\frac54$ and $s_2 \geq 0$. The authors also employed the almost conservation machinery of the ``$I$--method" to obtain GWP for   $s_1>-\frac47$ and $s_2 \geq 0$. Later in \cite{hadac}, Hadac obtained the same LWP result in a more general context.  The most recent improvement is  for data at the $-\frac54$ regularity, see \cite{li}.

The only work we are aware of on the initial-boundary value problems involving KP type equations is \cite{MP}. They considered the classical KP--II equation on a strip and obtained local weak solutions in certain weighted Sobolev spaces. In this paper, we study the fifth order KP--II equation on the half plane with initial and boundary data in  $L^2$ based Sobolev spaces and obtain low regularity strong solutions. We also impose a nonhomogenous boundary constraint at $y=0$.  

 Wellposedness of \eqref{kp2}  means local existence, uniqueness and continuity with respect to the initial data of distributional solutions. For the definition of the   usual Sobolev spaces  and their adapted generalization for the fifth order KP--II we refer the reader to the Notation subsection below. More precisely we have the following definition: 

\begin{definition}\label{def:lwp} Fix $s\in(0,\tfrac52)$. 
We say \eqref{kp2} is locally well--posed in $H^s(U)$, if \\
i) for any $g\in H^s(U)$ and $h\in \cH^s_{x,t}(U)$, with the compatibility condition $g(x,0)=h(x,0)$ a.e. for $s>\frac12$, the equation has a distributional solution 
$$
u\in   C^0_tH^s_{x,y}([0,T]\times U) \cap C^0_y\cH^s_{x,t}(\R^+ \times \R\times [0,T]),
$$
where $T=T(\|g\|_{H^s(U)},\|h\|_{\cH^s_{x,t} (U) })$, \\
ii) if $g_n\to g$ in $H^s(U)$ and $h_n\to h$ in $\cH^s_{x,t} (U)$, then $u_n\to u$ in the space above.
\end{definition}

Our first theorem establishes   local well--posedness.

\begin{theorem} \label{thm:local} Fix  $s\in (0, \frac52)\setminus\{\frac12\}$.  Then the equation \eqref{kp2} is locally well--posed in $H^s(U)$ in the sense of Definition~\ref{def:lwp}.
\end{theorem}

In addition we obtain the following smoothing estimate:
\begin{theorem} \label{thm:smooth} Fix $s\in (0, \frac52)\setminus \{\frac12\}$  and $a<\min(  \frac13,\frac{2s}3,\frac32-\frac{3s}5)$.  Then  for any $g\in H^s(U)$ and $h\in \cH^s_{x,t}(U)$,  with the additional compatibility condition $g(x,0)=h(x,0)$ a.e.  when $s>\frac12$,  the solution $u$ of \eqref{kp2}  satisfies
$$
u(x,t)-W_0^t(g,h)(x)\in C^0_tH^{s+a}_x([0,T]\times U),
$$
where $T$ is the local existence time, and  $ W_0^t(g,h)$ is the solution of the corresponding linear equation.
\end{theorem}

\begin{rmk}
We should note that the proof of  Theorem~\ref{thm:smooth} yields the analogous smoothing result for the problem on the full plane $\mathbb R^2$ which appears to be new.   
\end{rmk}

To study the half--plane problem we utilize the restricted norm method of Bourgain  \cite{bourgain,Bou2}. This continues our work initiated in  \cite{etnls}, \cite{etza} and \cite{egt}, of  establishing the regularity properties of nonlinear dispersive partial differential equations (PDE) on a half line  using the tools that are available in the case of the whole  line. We thus extend the data to the whole plane and use Laplace transform methods to   set up an equivalent integral equation (on $\R^2\times \R $) of the solution, see \eqref{eq:duhamel} below.  We then analyze the integral equation  using  the restricted norm method as in \cite{collianderkenig,etnls,etza} and multilinear $L^2$ convolution estimates. Our result is the first well--posedness result on the half--plane for a KP type dispersive equation.  Concerning uniqueness, the solution we obtain for the integral equation \eqref{eq:duhamel}  is unique. However, we cannot obtain a unique strong solution of the original PDE since our solution is a fixed point of \eqref{eq:duhamel}  that depends on the particular extension we use. We should also note that our method does  not immediately apply to  the initial boundary value problem for   the classical KP--II equation with the third order dispersion. We hope to address these two problems (uniqueness of solutions and well--posedness theory for KP--II)   in our future work. Another interesting problem is that of GWP for the fifth order KP--II. Unfortunately this is not as easy as in the full plane case since the presence of the boundary terms prevent certain energy identities to hold. Subsequently it is hard to obtain a priori bounds for our solutions in the Sobolev type norms we use in our local result.


We now discuss briefly the organization of the paper.  In Section 2,  we  introduce the appropriate function spaces, especially the  $X^{s,b}$ norm. 
We also construct the solutions of the linear problem and set up the Duhamel formula  for the full equation. The Duhamel formula incorporates the extension of the data on $\mathbb R^2$ and the evaluation of certain operators at the zero boundary. 
In Section 3, we obtain the a priori linear estimates that we need in order to put our solutions to the right function spaces. In Section 4, we prove the nonlinear estimates which is the main part of this paper. This section also provides the tools needed for the proof of Theorem~\ref{thm:smooth}.  In Section 5, we briefly outline the well--known process of establishing LWP and smoothing using the linear and nonlinear estimates of Sections 3 and 4. The last section, Section 6, is an Appendix where we state two calculus lemmas that we use throughout the paper. We finish the introduction with a notation subsection.

\subsection{Notation}

Recall that for $s\geq 0$, $H^s(\R^d)$ is defined as a subspace of $L^2$ via the norm
$$
\|f\|_{H^s}=\|f\|_{H^s(\R^d)}:=\Big(\int_{\R^d} \la \zeta\ra^{2s} |\widehat{f}(\zeta)|^2 d\zeta \Big)^{1/2},
$$
where $\la \zeta\ra:=(1+|\zeta|^2)^{1/2}$ and 
$$\widehat{f}(\zeta)=\cF f(\zeta)=\int_{\R^d} f(x)e^{-ix\cdot\zeta} dx$$ 
is the Fourier transform of $f$. We also set the notation
$$
f(\widehat{\zeta_j})=\cF_j f(\zeta_j)=\int_{\R} f(x)e^{-ix_j\zeta_j}dx_j
$$
for the Fourier transform in the $j$th space coordinate. The Laplace transform is defined as usual by
$$
\widetilde{f}(\lambda)=\cL f(\lambda)=\int_0^{\infty} f(t)e^{-\lambda t}dt,\quad \Re \lambda >0
$$
and in case of several variables we will write $f(\widetilde{\lambda_j})$ to represent the Laplace transform in a particular variable.

For a space time function $f$, we set the notation
$$
D_0f(x,t)=f(x,0,t).
$$

Throughout the paper we have $s\in(0,\frac52) $, $s\neq\frac12$.
We define $H^s(U)$ norm as
$$
\|g\|_{H^s(U)}:=\inf\big\{\|\tilde g\|_{H^s(\R^2)}: \tilde g(x,y)=g(x,y),\, y>0\big\}.
$$
We say $\tilde g$ is an $H^s(\R^2)$ extension of $g\in H^s(U)$ if $ \tilde g(x,y)=g(x,y)$ for $y>0$ and $\|\tilde g\|_{H^s(\R^2)}\leq 2 \|g\|_{H^s(U)}$.  
Note that, if $g  \in  H^s(U)$ for some $s>\frac12$, then by trace lemma  any $H^s$ extension  is in $C^0_yL^2_x$, and hence $g(x,0)$ is well defined as an $L^2$ function. 

The following Sobolev type space will be the natural choice for the boundary data we impose 
\begin{multline} \label{eq:cHs}
\cH^s=\cH^s_{x,t}(\R^2)=\Big\{\varphi:\R^2\to \R: \la \xi^2+\eta^2\ra^{\frac{s}{2}}\,\frac{\eta}{\xi}\widehat{\varphi}(\xi,\xi^5\pm \tfrac{\eta^2}{\xi})\in L^2_{\xi,\eta}\Big\}  \\
=\Big\{\varphi:\R^2\to \R: \la \xi^2+|\xi\beta-\xi^6|\ra^{\frac{s}{2}}\,\frac{|\xi\beta-\xi^6|^{\frac14}}{|\xi|^{\frac12}}\widehat{\varphi}(\xi,\beta)\in L^2_{\xi,\beta}\Big\},                                            
\end{multline}
for $s\in\R$. We define $\cH^s(U)$, where $U$ is the upper half plane, analogously.

Finally, we use $\la x,y\ra $ to denote $\la (x,y)\ra=\sqrt{1+x^2+y^2}$, and we reserve  the symbol $\mu $ for a smooth compactly supported function of time which is equal to $1$ on $[-1,1]$.

\section{Notion of a solution }
In order to construct the solutions of \eqref{kp2}, we first consider the linear problem with $g \in H^{s}(U)$ and $\chi_{t>0}h(x,t)\in \cH^s_{x,t}$:
\begin{equation}\label{kp2-linear}
\left\{
\begin{array}{l}
\partial_x\big(u_{t}-\D_x^5u \big)+  u_{yy}=0 \quad x \in {\R}, \, y>0,\, t>0,\\
u(x,y,0)=g(x,y),\quad u(x,0,t)=h(x,t). \\
\end{array}
\right.
\end{equation}
We denote by $g_e$ an $H^s(\R^2)$ extension of $g\in H^s(U)$ with $\|g_e\|_{H^s}\lesssim \|g\|_{H^s(U)}$. With this notation the unique solution of \eqref{kp2-linear} for $0\leq t\leq 1$ is the restriction of 
$$u(t)=W_0^t(g,h)=W_{\R^2}(t)g_e+W_0^t(0,h-p) $$
to $U$, where the first summand is the free fifth order KP--II propagator 
\begin{multline*}
W_{\R^2}(t)g_e(x,y)=e^{-tK_2}g_e(x,y)=\cF^{-1} [e^{i(\xi^5-\frac{\eta^2}{\xi})t}\widehat{g_e}(\xi,\eta)](x,y)\\
=\int_{\R^2} e^{i\xi x}e^{i\eta y}e^{i(\xi^5-\frac{\eta^2}{\xi})t}\widehat{g_e}(\xi,\eta) d\xi d\eta,
\end{multline*}
where $K_2=-\D_x^5+\D_x^{-1}\D_y^2$. In addition, we have 
$$
p(x,t)=\mu(t)D_0(W_{\R^2}g_e)=\left.\mu(t)W_{\R^2}(t)g_e(x,y)\right|_{y=0},
$$
with $\mu(t)$ being a smooth function that is compactly supported and equals $1$ on $[-1,1]$.

Calculation of $W_0^t(0,h)$ follows from taking the Fourier transform in $x$ and Laplace transform in $t$ of the linear fifth order KP--II problem \eqref{kp2-linear} with $g=0$
$$
\left\{
\begin{array}{l}
\D_y^2 u(\widehat{\xi},y,\widetilde{\lambda})+(i\lambda \xi+\xi^6)u(\widehat{\xi},y,\widetilde{\lambda})=0,\quad  y>0,\\
u(\widehat{\xi},0,\widetilde{\lambda})=h(\widehat{\xi},\widetilde{\lambda}),\quad u(\widehat{\xi},\infty,\widetilde{\lambda})=0. \\
\end{array}
\right.
$$   
The solution of this is obtained as
$$
u(\widehat{\xi},y,\widetilde{\lambda})=e^{(-i\lambda \xi-\xi^6)^{1/2}y}h(\widehat{\xi},\widetilde{\lambda}), \quad \Re (-i\lambda \xi-\xi^6)^{1/2}<0\ \text{and}\ \Re \lambda >0.
$$
Setting $r(\xi,\lambda):=(-i\lambda \xi-\xi^6)^{1/2}$ with $\Re (-i\lambda \xi-\xi^6)^{1/2}<0$, and $\lambda=\gamma+i\beta$ with $\gamma>0$, we analyze the suitable branches of square root and find 
$$
\lim_{\gamma\to 0^+} r(\xi,\lambda)=\left\{\begin{array}{rl} -|\beta\xi-\xi^6|^{1/2},&\quad \beta\xi-\xi^6>0\\
                                                             i|\beta\xi-\xi^6|^{1/2},&\quad \beta\xi-\xi^6<0\ \text{and}\ \xi>0\\
                                                            -i|\beta\xi-\xi^6|^{1/2},&\quad \beta\xi-\xi^6<0\ \text{and}\ \xi<0. \end{array}\right.
$$ 
Using this and letting $\gamma\to 0^+$, we obtain the solution \eqref{kp2-linear} with $g=0$ by 
$$u(x,y,t)=\cF^{-1}_1 \cL^{-1}[u(\cdot,y,\cdot)](x,t).$$ 
We remark that in the resulting integrals we perform algebraic manipulations including changes of variables of the type $\beta\xi-\xi^6=\pm \eta^2$  and pass from Laplace to Fourier transform in time and end up with the following solution
\begin{multline}\label{linear-soln-g0}
W_0^t(0,h)(x,y)=\iint\limits_{\R^-\times\R^-\, \cup \,\R^+\times \R^+} e^{i\xi x +i \eta y} e^{i(\xi^5-\frac{\eta^2}{\xi})t}\,\frac{2\eta}{\xi}\widehat{h\chi_{t>0}}(\xi,\xi^5-\tfrac{\eta^2}{\xi})d\eta d\xi\\
+\int\limits_{\R} \int\limits_{\R^+}e^{i\xi x - \eta y} e^{i(\xi^5+\frac{\eta^2}{\xi})t}\,\frac{2\eta}{|\xi|}\widehat{h\chi_{t>0}}(\xi,\xi^5+\tfrac{\eta^2}{\xi})d\eta d\xi\\=:W_1h(x,y,t)+W_2h(x,y,t). 
\end{multline}
Note that $W_1$ is now well defined for every $x,y$ and $t$ in $\R$. We extend $W_2$ to all $y$ by multiplying by a smooth function $\rho$ supported on $(-2,\infty)$ that is equal to $1$ on $(0,\infty)$, i.e., 
\begin{multline}\label{w2}
W_2h(x,y,t)=\int\limits_{\R} \int\limits_{\R^+} \rho(\eta y)\,e^{i\xi x - \eta y} e^{i(\xi^5+\frac{\eta^2}{\xi})t}\,\frac{2\eta}{|\xi|}\widehat{h\chi_{t>0}}(\xi,\xi^5+\tfrac{\eta^2}{\xi})d\eta d\xi \\ =\int\limits_{\R} \int\limits_{\R^+}f(\eta y) e^{i\xi x} e^{i(\xi^5+\frac{\eta^2}{\xi})t}\,\frac{2\eta}{|\xi|}\widehat{h\chi_{t>0}}(\xi,\xi^5+\tfrac{\eta^2}{\xi})d\eta d\xi.                                                
\end{multline}
Here $f(y)=\rho(y)e^{-y}$ is a Schwartz function. 
In order for the solution \eqref{linear-soln-g0} above to make sense we require $\chi_{t>0}h\in \cH^s_{x,t}$ where 
\begin{multline*}
\cH^s=\cH^s_{x,t}(\R^2)=\Big\{\varphi:\R^2\to \R: \la \xi^2+\eta^2\ra^{\frac{s}{2}}\,\frac{\eta}{\xi}\widehat{\varphi}(\xi,\xi^5\pm \tfrac{\eta^2}{\xi})\in L^2_{\xi,\eta}\Big\}  \\
=\Big\{\varphi:\R^2\to \R: \la \xi^2+|\xi\beta-\xi^6|\ra^{\frac{s}{2}}\,\frac{|\xi\beta-\xi^6|^{\frac14}}{|\xi|^{\frac12}}\widehat{\varphi}(\xi,\beta)\in L^2_{\xi,\beta}\Big\},                                            
\end{multline*}
for $s\in\R$.  In the next section we will prove a Kato smoothing estimate, see Proposition~\ref{freekp2:y=0}, that implies that the space $\cH^s(U)$ is the natural choice
for the boundary data.
     
 We now establish embedding and extension properties of these spaces.       
\begin{lemma}\label{lem:embedding}
For $s>1$, the space $\mathcal H^s_{x,t}$ embeds continuously into $C^0_{x,t}$. Moreover, for $s>\frac12$, we have the trace lemma; the space $\mathcal H^s_{x,t}$ embeds continuously into $C^0_tL^2_x$, and  in particular $\sup_t \|\varphi\|_{ L^2_x}\les \| \varphi\|_{\mathcal H^s_{x,t}}$.
\end{lemma}
\begin{proof}
For the first claim, it suffices to prove that given $\varphi \in \mathcal H^s_{x,t}$,
$\|\widehat\varphi\|_{L^1(\R^2)}\les \| \varphi\|_{\mathcal H^s_{x,t}}$.
By the Cauchy--Schwarz inequality this follows from
$$
\int_{\R^2}\frac{|\xi| }{\la\xi^2+|\xi\beta-\xi^6|\ra^s|\xi\beta-\xi^6|^{\frac12}}d\xi d\beta
=\int_{\R^2}\frac{1 }{\la\xi^2+\rho^2\ra^s} d\xi d\rho<\infty
$$
since $s>1$. In the first equality we used the change of variable $\rho^2=|\xi\beta-\xi^6|$ in the $\beta$ integral. 

Similarly, for the second claim it suffices to prove that
$\|\widehat\varphi\|_{L^2_\xi L^1_\beta }\les \| \varphi\|_{\mathcal H^s_{x,t}}$. By the Cauchy--Schwarz inequality in the $\beta$ integral and the same change of variable this follows from
$$
\sup_\xi \int_{\R}\frac{|\xi| }{\la\xi^2+|\xi\beta-\xi^6|\ra^s|\xi\beta-\xi^6|^{\frac12}}  d\beta
=\sup_\xi  \int_{\R }\frac{1 }{\la\xi^2+\rho^2\ra^s}  d\rho\les 1,
$$
since $s>\frac12$.
\end{proof}
 
\begin{lemma}\label{lem:extend}
For $-\frac32<s<\frac12$, we have 
$$
\|\chi_{t>0} \varphi(x,t)\|_{\mathcal H^s_{x,t}(\R^2)}\les \| \varphi \|_{\mathcal H^s (U)}.
$$ 
Moreover, for $\frac12<s<\frac52$, we have the same bound provided that the trace $\varphi(x,0)$ is zero.
\end{lemma}
\begin{proof}
Since $\mF\big(\chi_{t>0}  \varphi\big)(\xi,\beta) = H \widehat\varphi(\xi,\beta),$
where $H$ is essentially the Hilbert transform in the $\beta$ variable:
$$ Hf(\xi,\beta)= \mF_t\big(\chi_{t>0} f(\xi,t^\vee)\big)(\beta),$$ 
It suffices to prove that $m(\xi,\beta)= \la \xi^2+|\xi\beta-\xi^6|\ra^{s}\,\frac{|\xi\beta-\xi^6|^{\frac12}}{|\xi| }$ is an $A_2$ weight in $\beta$ uniformly in $\xi$ for $-\frac32<s<\frac12$, see \cite{JD}. We first note that $\omega(\beta)=\la \beta\ra^s |\beta|^{\frac12}$ is an $A_2$ weight for $-\frac32<s<\frac12$.
Recalling that the $A_2$ constant is invariant under dilations, translations and scaling, we can replace $m$ with $ \la \xi^2+| \beta |\ra^{s} | \beta |^{\frac12}$. Noting that for $|\xi|<1$, we can further simplify $m$ to $\omega$, the statement follows in this case. For $|\xi|>1$, we can consider 
$( \xi^2+| \beta |)^{s} | \beta |^{\frac12}$, which once again boils down to $\omega$ by scaling and dilating. 

For the second part, we note that 
$\|\chi_{t>0}  \varphi \|_{\mathcal H^s_{x,t}(\R^2)} =\|T\big(\chi_{t>0}  \varphi \big)\|_{\mathcal H^{s-2}_{x,t}(\R^2)}$, where $T$ is the multiplier operator with the multiplier $1+\xi^2+|\xi\beta-\xi^6|$. Furthermore
$$
\|T\big(\chi_{t>0}   \varphi\big)\|_{\mathcal H^{s-2}_{x,t}(\R^2)}\leq \|T_1\big(\chi_{t>0}   \varphi\big)\|_{\mathcal H^{s-2}_{x,t}(\R^2)} + \|T_2\big(\chi_{t>0} \varphi\big)\|_{\mathcal H^{s-2}_{x,t}(\R^2)},
$$
where $T_1$ and $T_2$ are multiplier operators with multipliers $1+\xi^2$ and $\xi\beta-\xi^6$, respectively. Since $\varphi$ has trace zero, $\partial_t\big(\chi_{t>0}  \varphi\big) =\chi_{t>0}   \partial_t  \varphi$ in the sense of distributions. Therefore, we have   $T_j\big(\chi_{t>0}   \varphi\big)= \chi_{t>0}  T_j\varphi$, $j=1,2$. Also using the first part for $s-2$, we obtain
 $$
\|T\big(\chi_{t>0}   \varphi\big)\|_{\mathcal H^{s-2}_{x,t}(\R^2)}\les \|T_1  \varphi \|_{\mathcal H^{s-2}_{x,t}(U)} + \|T_2   \varphi \|_{\mathcal H^{s-2}_{x,t}(U)} \les \|\varphi \|_{\mathcal H^{s}_{x,t}(U)}.
$$
\end{proof}

We now consider the integral equation
\begin{equation}\label{eq:duhamel}
u(t)=\mu(t)W_0^t(g,h)  + \mu(t) \int_0^tW_{\R^2}(t- t^\prime)  F(u) \,d t^\prime - \mu(t) W_0^t\big(0, q  \big)(t),
\end{equation}
where
$$
F(u)= -\mu(t/T) uu_x   \text{ and }  q(t)=\mu(t ) D_0\Big(\int_0^tW_\R(t- t^\prime)  F(u)\, d t^\prime \Big).
$$ 
In what follows we will prove that the  integral equation \eqref{eq:duhamel} has a unique solution in the $X^{s,b} $ space \eqref{def:xsb} on $\R^2\times \R$ for some $T<1$. The a priori linear estimates in Section~\ref{sec:lin} will guarantee that the solution also belongs to 
$  C^0_tH^s_{x,y}([0,T]\times U) \cap C^0_y\cH^s_{x,t}(\R^+ \times \R\times [0,T]),$ and that it depends  continuously  on data in these spaces, see Section~\ref{sec:lwp}.  
Using the definition of the boundary operator, it is clear that the restriction of $u$ to $U\times [0,T]$ satisfies  \eqref{kp2}  in the distributional sense. Also note that the smooth solutions of \eqref{eq:duhamel} satisfy \eqref{kp2}  in the classical sense.

The Bourgain spaces, $X^{s,b}(\R^2\times\R)$ (see \cite{bourgain,Bou2}), will be defined as the closure of compactly supported smooth functions under the norm 
\be\label{def:xsb}\|u\|_{X^{s,b}} :=\|\langle \tau-\xi^5+\tfrac{\eta^2}{\xi} \rangle^{b} \langle \xi^2+\eta^2 \rangle^{\frac{s}{2}}\,\widehat{u}(\xi,\eta,\tau)\|_{L_{\xi,\eta,\tau}^2}.  
\ee  
We  recall the embedding $X^{s,b}\subset C^0_t H^{s}_{x,y}$ for $b>\frac{1}{2}$ and the following inequalities from \cite{bourgain,gtv,etbook}. 

For any $s,b$ we have
\begin{equation}\label{eq:xs1}
\|\eta(t)W_{\R^2} g\|_{X^{s,b}}\les \|g\|_{H^s}.
\end{equation}
For any $s\in \mathbb R$,  $0\leq b_1<\frac12$, and $0\leq b_2\leq 1-b_1$, we have
\begin{equation}\label{eq:xs2}
\Big\| \eta(t) \int_0^t W_\R(t-t^\prime)  F(t^\prime ) dt^\prime \Big\|_{X^{s,b_2} }\lesssim   \|F\|_{X^{s,-b_1} }.
\end{equation}
Moreover, for $T<1$, and $-\frac12<b_1<b_2<\frac12$, we have
\begin{equation}\label{eq:xs3}
\|\eta(t/T) F \|_{X^{s,b_1}}\les T^{b_2-b_1} \|F\|_{X^{s,b_2}}.
\end{equation}

 \section{A priori linear estimates} \label{sec:lin}
 In this section we provide a priori estimates for the linear terms in \eqref{eq:duhamel}. 
We start with the following Kato smoothing type statement  involving the $\cH^s$ norm for the linear group. This  estimate and the Proposition~\ref{prop:cHstoHs} below justify the choice of $\cH^s$ space    in Definition~\ref{def:lwp}.                                                          
\begin{prop}\label{freekp2:y=0} For $s\geq 0$, and $g\in H^s(\R^2)$, we have
$\mu(t) W_{\R^2}g \in C^0_y \cH^s_{x,t}$, and we have
$$\|\mu(t) W_{\R^2}g \|_{L^\infty_y \cH^s_{x,t}}\les \|g\|_{H^s}.
$$
\end{prop}
\begin{proof}
For short, set $W(x,y,t) =\mu(t) W_{\R^2}g (x,y,t)$. Taking the Fourier transform in $x$ we get
$$
 W(\widehat{\xi},y,t)=\int_{\R} \mu(t) e^{i(\xi^5-\frac{\theta^2}{\xi})t} e^{i\theta y}\widehat{g}(\xi,\theta)d\theta.
$$
Now the Fourier transform in $t$ gives
$$
W(\widehat{\xi},y,\widehat{\eta})  =\int_{\R} \widehat{\mu}(\eta-\xi^5+\tfrac{\theta^2}{\xi})e^{i\theta y} \widehat{g}(\xi,\theta)d\theta.
$$ 
By dominated convergence theorem, the statement follows from the claim:
$$
I:=\int_{\R^2} \la \xi^2+\eta^2\ra^s\, \frac{\eta^2}{\xi^2} \Big(\int_{\R} |\widehat{\mu}(\tfrac{\theta^2\pm\eta^2}{\xi})| |\widehat{g }(\xi,\theta)|d\theta\Big)^2 d\eta d\xi\les \|g\|_{H^s}^2.
$$
Applying the Cauchy--Schwarz inequality to the $\theta$-integral we get
$$
\Big(\int_{\R} |\widehat{\mu}(\tfrac{\theta^2\pm\eta^2}{\xi})| |\widehat{g }(\xi,\theta)|d\theta\Big)^2 \leq \left\|\widehat{\mu}(\tfrac{\theta^2\pm\eta^2}{\xi})\right\|_{L^1_{\theta}} \int_{\R} |\widehat\mu (\tfrac{\theta^2\pm\eta^2}{\xi})| |\widehat{g }(\xi,\theta)|^2d\theta,
$$
and using the fact that for any $M>0$
$$
|\widehat{\mu}(\tfrac{\theta^2\pm\eta^2}{\xi})|\lesssim \frac{1}{\big\la\frac{ \theta^2\pm\eta^2 }{ \xi }\big\ra^{M}},
$$
we have the estimate
$$
\left\|\widehat{\mu}(\tfrac{\theta^2\pm\eta^2}{\xi})\right\|_{L^1_{\theta}}\lesssim \int_{\R}  \frac{d\theta}{\big\la\frac{ \theta^2\pm\eta^2 }{ \xi }\big\ra^{M}}\lesssim \int_0^{\infty} \frac{|\xi|^{1/2}d\rho}{\rho^{1/2}\la\rho\pm \frac{\eta^2}{\xi}\ra^M}\lesssim \frac{|\xi|^{1/2}}{\la \frac{\eta^2}{\xi}\ra^{1/2}},
$$
where we used  the change of variable $\rho=\frac{\theta^2}{|\xi|}$ and then Lemma~\ref{lem:sums} with $M>1$.

We now combine these estimates to bound the integral $I$ as follows
$$
I\lesssim \int_{\R^3} \la \xi^2+\eta^2 \ra^s \frac{|\eta| |\xi|^{ -1}}{\big\la\frac{ \theta^2\pm\eta^2 }{ \xi }\big\ra^{M}} |\widehat{g }(\xi,\theta)|^2 d\eta d\theta d\xi.
$$
It suffices to consider only $|\theta^2-\eta^2|$ case of $|\theta^2\pm\eta^2|$ in the denominator and show that 
$$
J:=\int_{\R} \la \xi^2+\eta^2 \ra^s \frac{|\eta| |\xi|^{ -1}}{\big\la\frac{ \theta^2-\eta^2 }{ \xi }\big\ra^{M}} d\eta \lesssim \la \xi^2+\theta^2 \ra^s.
$$
We consider two cases to prove the desired bound for $J$.

\noindent
\textsc{Case 1:} $|\eta|\les |\theta|$

In this region we have
$$
J\les \int  \la \xi^2+\theta^2 \ra^s \frac{|\eta| d\eta}{|\xi| \Big(1+\frac{|\theta^2-\eta^2|}{|\xi|}\Big)^M}\lesssim \int_{\R} \la \xi^2+\theta^2 \ra^s \frac{|\rho| d\rho}{(1+\rho^2)^M}\lesssim \la \xi^2+\theta^2 \ra^s,
$$
for $M>1$, where $\rho^2=\frac{|\eta^2-\theta^2|}{|\xi|}$.

\noindent
\textsc{Case 2:}   $|\eta|>2  |\theta|$

By $|\theta^2-\eta^2|\gtrsim \eta^2$ and again $\rho^2=\frac{\eta^2}{|\xi|}$ we have
\begin{multline*}
J\lesssim \int_{|\rho|>2\frac{|\theta|}{\sqrt{|\xi|}}} \la \xi^2+|\xi|\rho^2 \ra^s \frac{|\rho| d\rho}{(1+\rho^2)^M}\lesssim \sup_{|\rho|>2\frac{|\theta|}{\sqrt{|\xi|}}} \frac{\la \xi^2+|\xi|\rho^2 \ra^s}{(1+\rho^2)^{M/2}}\,\int_{\R} \frac{|\rho|d\rho}{(1+\rho^2)^{M/2}} \\ 
\lesssim \sup_{|\rho|>2\frac{|\theta|}{\sqrt{|\xi|}}} \frac{1+|\xi|^{2s}+|\xi|^s|\rho|^{2s}}{(1+\rho^2)^{M/2}} \lesssim 1+|\theta|^{2s}+|\xi|^{2s} \lesssim \la \xi^2+\theta^2 \ra^s,
\end{multline*}
for $M>\max(1,2s)$.  
\end{proof}
Next, we establish a priori estimates for the boundary operator in $\cH^s$ spaces:
\begin{prop}\label{prop:verttrace} Fix $s\geq 0$. We have 
$\mu(t)W_1h(x,y,t)$ and $\mu(t)W_2h(x,y,t)\in C^0_y\cH^s_{x,t}$ for $\chi_{t>0}h(x,t)\in\cH^s_{x,t}$. 
\end{prop}

\begin{proof}
Both follow from the proof of Proposition~\ref{freekp2:y=0}. We just note
\begin{multline*}
\|\mu W_2h(y)\|_{\cH^s_{x,t}}^2=\int_{\R^2} \la \xi^2+\eta^2\ra^s\, \frac{\eta^2}{\xi^2} \left|\int_{\R^+} \widehat{\mu}(\tfrac{-\theta^2\pm\eta^2}{\xi})f(\theta y)\widehat{\psi}(\xi,\theta)d\theta\right|^2 d\eta d\xi\\
\lesssim \int_{\R^2}\int_{\R^+} \la \xi^2+\eta^2\ra^s\, \frac{\eta^2}{\xi^2} \left\|\widehat{\mu}(\tfrac{-\theta^2\pm\eta^2}{\xi})\right\|_{L^1_{\theta}} |\widehat{\mu}(\tfrac{-\theta^2\pm\eta^2}{\xi})| |\widehat{\psi}(\xi,\theta)|^2d\theta d\eta d\xi,
\end{multline*}
and that $\widehat{\psi}(\xi,\eta)=\frac{2\eta}{|\xi|}\widehat{\chi_{t>0}h}(\xi,\eta)$ which yields $\psi\in H^s$.

Continuity in $y$ is a consequence of the dominated convergence theorem and $f\in L^1$. 
\end{proof} 
\begin{prop}\label{prop:cHstoHs} Fix $s\geq 0$. We have 
$\mu(t)W_1h(x,y,t)$ and $\mu(t)W_2h(x,y,t)\in C^0_tH^s_{x,y}$ for $\chi_{t>0}h(x,t)\in\cH^s_{x,t}$. 
\end{prop}
\begin{proof} Note that  
$$
W_1h(\widehat\xi,\widehat\eta,t)=e^{i(\xi^5-\frac{\eta^2}{\xi})t}\,\frac{2\eta}{\xi}\widehat{h\chi_{t>0}}(\xi,\xi^5-\tfrac{\eta^2}{\xi}) \chi_{\xi\eta>0}.
$$
Therefore, the claims immediately follows from the definition of $\cH^s$ norm and the dominated convergence theorem.

For $W_2$, we first consider the case $s=0$. 
Note that by Plancherel in the $x$ variable, we have 
$$\|W_2h\|_{L^2_{x,y}}^2\les  \Big\|\int f(\eta y)\Big|\frac{ \eta}{\xi}\widehat{h\chi_{t>0}}(\xi,\xi^5+\tfrac{\eta^2}{\xi})\Big|  d\eta \Big\|_{L^2_{\xi,y}}^2\leq
\Big\|\int f(\eta y) g(\eta) d\eta \Big\|_{L^2_{y}}^2, $$
where $g(\eta)=\|\frac{ \eta}{\xi}\widehat{h\chi_{t>0}}(\xi,\xi^5+\tfrac{\eta^2}{\xi}) \|_{L^2_\xi}.$ 
Noting that $\|g\|_{L^2}\les \|\chi_{t>0}h(x,t) \|_{\cH^0_{x,t}}$, the statement follows from the $L^2$ boundedness of the operator 
$$
Tg(y):=\int f(\eta y) g(\eta) d\eta,$$
which was proved in Lemma 3.2 of \cite{etnls}. The statement for $s>0$ follows from this and interpolation as described in Lemma 3.2 of \cite{etnls}.
\end{proof}


Since we will run the fixed point argument in Bourgain spaces we now obtain estimates for the boundary operator in $X^{s,b}$. First  recall that $\mu(t)W_{\R^2} g_e(x,y,t)\in X^{s,b}$ for every $b\in\R$ and $s \geq0$, see e.g. \cite{etbook}.
Upon this,  with
$$\widehat{\psi}(\xi,\eta)=\frac{2\eta}{\xi}\widehat{h\chi_{t>0}}(\xi,\xi^5-\frac{\eta^2}{\xi}),$$
we easily see that
$$
\|\mu W_1 h\|_{X^{s,b }}\leq \|\mu W_{\R^2} \psi\|_{X^{s,b }}\lesssim \|\psi\|_{H^s}\lesssim \|\chi_{t>0}h\|_{\cH^s_{x,t}},
$$
which entails $\mu W_1 h(x,y,t)\in X^{s,b}$. 
\begin{prop}\label{w2:xsbb1}
Let $b\leq\frac12$ and $s \geq 0$. Then for $h$ satifying $\chi_{t>0}   h\in \cH^s_{x,t}(\R^2)$,  we have 
$$\|\mu(t) W_2 h\|_{ X^{s,b }}\les \|\chi_{t>0}  h\|_{\cH^s_{x,t}(\R^2)}.$$
\end{prop}

\begin{proof}
From \eqref{w2}, we compute
$$
\widehat{\mu W_2 h}(\xi,\theta,\tau)=2\int_0^{\infty} \widehat{\mu}(\tau-\xi^5-\tfrac{\eta^2}{\xi})\frac{\widehat{f}(\theta/\eta)}{|\xi|}\widehat{h}(\xi,\xi^5+\tfrac{\eta^2}{\xi})d\eta.
$$
It is enough to prove the statement for $b=\frac12$ and $s=0$, for $s>0$ interpolation yields the desired result. Now
$$
\|\mu W_2 h\|_{X^{0,\frac12 }} = \Big\|\langle \tau-\xi^5+\tfrac{\theta^2}{\xi} \rangle^{ \frac12}  \int_0^{\infty} \widehat{\mu}(\tau-\xi^5-\tfrac{\eta^2}{\xi})\frac{\widehat{f}(\theta/\eta)}{\eta}\widehat{\psi}(\xi,\eta)d\eta \Big\|_{L^2_{\tau,\theta,\xi}},
$$
where $\widehat{\psi}(\xi,\eta)=\frac{2\eta}{|\xi|}\widehat{h\chi_{t>0}}(\xi,\xi^5+\frac{\eta^2}{\xi}) \in L^2$. Using
$$
\langle \tau-\xi^5+\tfrac{\theta^2}{\xi} \rangle \lesssim \langle \tau-\xi^5-\tfrac{\eta^2}{\xi} \rangle \langle \tfrac{\eta^2+\theta^2}{\xi} \rangle
$$
and the Schwarz decay of $\widehat\mu$,   we obtain
$$
\|\mu W_2 h\|_{X^{0,b }}\lesssim \Big\| \int_0^{\infty} \frac{\langle \tfrac{\eta^2+\theta^2}{\xi} \rangle^{\frac12 }}{\la \tau-\xi^5-\frac{\eta^2}{\xi} \ra^2} \frac{|\widehat{f}(\theta/\eta)|}{\eta}|\widehat{\psi}(\xi,\eta)|d\eta\Big\|_{L^2_{\tau,\theta,\xi}}.
$$
We recall that 
$$
\Big|\frac{1}{\eta}\widehat{f}(\theta/\eta)\Big|\lesssim \frac{1}{|\eta|} \la \theta/\eta \ra^{-2 }=\frac{|\eta| }{ \eta^2+\theta^2 }
$$
and thereby reach the bound
$$
\|\mu W_2 h\|_{X^{0,\frac12 }}\lesssim \Big\| \int_0^{\infty} \frac{\langle \tfrac{\eta^2+\theta^2}{\xi} \rangle^{\frac12}}{\la \tau-\xi^5-\frac{\eta^2}{\xi} \ra^2} \frac{|\eta| }{ \eta^2+\theta^2 }|\widehat{\psi}(\xi,\eta)|d\eta\Big\|_{L^2_{\tau,\theta,\xi}}:=I.
$$

We estimate the integral $I$ on the right hand side in two regions separately.

\noindent
\textsc{Case 1:} $\eta^2+\theta^2<|\xi|$.

In this case we find
\begin{multline*}
I\lesssim \Big\| \int_0^{\infty} \frac{|\widehat{\psi}(\xi,\eta) |\chi_{\eta^2+\theta^2<|\xi|}}{\la \tau-\xi^5 \ra^2} \frac{|\eta| }{ \eta^2+\theta^2 }d\eta\Big\|_{L^2_{\tau,\theta,\xi}}\\ \lesssim \Big\| \int_0^{\infty} |\widehat{\psi}(\xi,\eta)| \frac{|\eta| }{ \eta^2+\theta^2 }\chi_{\eta^2+\theta^2<|\xi|}d\eta\Big\|_{L^2_{\theta,\xi}},
\end{multline*}
upon evaluating $L^2_{\tau}$ norm. Define the kernel
$$
K_{\xi}(\theta,\eta):=\frac{|\eta| }{ \eta^2+\theta^2 }\chi_{\eta^2+\theta^2<|\xi|},
$$
and therefore the operator $T_\xi$ on $L^2$ by
$$
(T_\xi\widehat{\psi})( \theta):=\int_0^{\infty} K_{\xi}(\theta,\eta)\widehat{\psi}(\xi,\eta)d\eta.
$$
Now we have $I\lesssim \|T_\xi\widehat{\psi}\|_{L^2_{\theta,\xi}}$. We apply Lemma~\ref{lem:schur} with $q(\eta)=|\eta|^{-\epsilon}$ and $p(\theta)=|\theta|^{-\epsilon}$, for suitable $\epsilon>0$. With this we observe that
$$
\int_0^{\infty} \frac{K_{\xi} (\theta,\eta)}{|\eta|^{\epsilon}}d\eta \leq \int_0^{\infty} \frac{|\eta| }{ \eta^2+\theta^2 } \frac{d\eta}{|\eta|^{\epsilon}}\leq \frac{1}{|\theta|^{\epsilon}} \int_0^{\infty} \frac{|\rho| d\rho}{(1+\rho^2)  |\rho|^{\epsilon}}\lesssim \frac{1}{|\theta|^{\epsilon}},
$$
and that
$$
\int_{-\infty}^{\infty} \frac{K_{\xi} (\theta,\eta)}{|\theta|^{\epsilon}}d\theta\leq \frac{1}{|\eta|^{\epsilon}}\int_{-\infty}^{\infty} \frac{d\rho}{(1+\rho^2)  |\rho|^{\epsilon}}\lesssim \frac{1}{|\eta|^{\epsilon}},
$$
provided $\epsilon <1$. Consequently we arrive at $\|T_\xi\widehat{\psi}\|_{L^2_{\theta }}\lesssim \|\widehat{\psi}\|_{L^2_{ \eta}}$ uniformly in $\xi$.
And hence, $\|T_\xi\widehat{\psi}\|_{L^2_{\theta,\xi}}\les \|\widehat{\psi}\|_{L^2_{ \eta,\xi}}<\infty$.

\noindent
\textsc{Case 2:} $\eta^2+\theta^2\geq|\xi|$. 

In this region the integral $I$ is bounded as
\begin{multline*}
I\lesssim \Big\| \int_0^{\infty} \frac{|\widehat{\psi}(\xi,\eta)| \chi_{\eta^2+\theta^2\geq|\xi|}}{\la \tau-\xi^5-\frac{\eta^2}{\xi} \ra^2} \frac{|\eta|  }{|\xi|^{\frac12}(\eta^2+\theta^2)^{\frac12}}d\eta\Big\|_{L^2_{\tau,\theta,\xi}}\\ \leq
\Big\| \int_0^{\infty} \frac{|\widehat{\psi}(\xi,\eta)|  |\eta| }{\la \tau-\xi^5-\frac{\eta^2}{\xi} \ra^2|\xi|^\frac12} \Big\|\frac{1}{(\eta^2+\theta^2)^{\frac12}}\Big\|_{L^2_{\theta}}d\eta\Big\|_{L^2_{\tau,\xi}}\\
\les \Big\| \int_0^{\infty} \frac{|\widehat{\psi}(\xi,\eta)|  |\eta|^\frac12 }{\la \tau-\xi^5-\frac{\eta^2}{\xi} \ra^2|\xi|^\frac12}  d\eta\Big\|_{L^2_{\tau,\xi}}.
\end{multline*}
We set $\eta^2=|\xi\rho|$ and $\tau^{\prime}=\tau-\xi^5$ to find
$$
I\lesssim \Big\| \int_{-\infty}^{\infty} \frac{\widehat{\psi}(\xi,\sqrt{|\xi\rho|})}{\la \tau^{\prime}-\rho \ra^2} \frac{ |\xi|^{\frac{1}{4}}}{|\rho|^{\frac{1}{4}}}d\rho\Big\|_{L^2_{\tau^{\prime},\xi}},
$$
and then apply Young's inequality to the $L^2_{\tau^{\prime}}$ norm  and reinstate the $\eta$ variable. This brings the following bound
$$
I\lesssim \Big\| \Big(\int_{-\infty}^{\infty}  \frac{ |\xi|^{\frac{1}{2}}}{|\rho |^{ \frac12}} |\widehat{\psi}(\xi,\sqrt{|\xi\rho|})|^2 d\rho\Big)^{\frac{1}{2}}\Big\|_{L^2_{\xi}}\lesssim   \|\widehat{\psi}\|_{L^2_{\xi,\eta}}.
$$ 
\end{proof}

 The following   is a Kato smoothing type estimate for the nonlinear Duhamel term:
\begin{prop}\label{prop:nonlinKato}
For any compactly supported smooth function $\mu$, we have
$$\mu(t) \int_0^t W_{\R^2}(t-t^{\prime})F(x,y,t^{\prime})dt^{\prime}\in C^0_y \cH^s_{x,t}$$
with the norm bound
$$\Big\| \mu(t) \int_0^t W_{\R^2}(t-t^{\prime})F(x,y,t^{\prime})dt^{\prime}\Big\|_{\cH^s_{x,t}}\lesssim \|F\|_{X^{s,-b}},$$
for $0\leq s \leq \frac12 $ and $b<\frac{1}{2}$. For $s>\frac12$ we have the bound
\be\label{weird}\Big\| \mu(t) \int_0^t W_{\R^2}(t-t^{\prime})F(x,y,t^{\prime})dt^{\prime}\Big\|_{\cH^s_{x,t}}   \les \|F\|_{X^{s,-b}} + \|F\|_{X^{\frac12+, \frac{s}2-\frac14,-b_1}},
\ee
where  $b<\frac{1}{2}$, $b_1=\frac34-\frac{s}2$, and
$$
\|F\|_{X^{\frac12+, \frac{s}2-\frac14,-b_1}}=\big\|\la\xi,\theta\ra^{\frac12+} |\xi|^{ \frac{s}2-\frac14} 
\la\lambda-\xi^5+\tfrac{\theta^2}\xi\ra^{-b_1} \widehat F(\xi,\theta,\lambda)\big\|_{L^2_{\xi,\theta,\lambda}}.
$$
\end{prop}

\begin{proof}
We first note that continuity in $y$ follows from the dominated convergence theorem. Secondly, it is enough to show the assertion for 
$$\mu(t) D_0 \int_0^t W_{\R^2}(t-t^{\prime})F(x,y,t^{\prime})dt^{\prime}=\mu(t) \Big[\int_0^t W_{\R^2}(t-t^{\prime})F(x,y,t^{\prime})dt^{\prime}\Big]_{y=0}$$
by translation invariance of $X^{s,b}$ norms in space variable. Now this quantity explicitly is
$$
\mu(t) D_0 \int_0^t W_{\R^2}(t-t^{\prime})F(x,y,t^{\prime})dt^{\prime}=\mu(t) \int_{\R^2} \int_0^t e^{i\xi x} e^{i(\xi^5-\frac{\eta^2}{\xi})(t-t^{\prime})} F(\widehat{\xi},\widehat{\eta},t^{\prime})dt^{\prime} d\xi d\eta,
$$
where
$$
F(\widehat{\xi},\widehat{\eta},t^{\prime})=\int_{\R} e^{i\lambda t^{\prime}} \widehat{F}(\xi,\eta,\lambda)d\lambda.
$$
Using this we evaluate the resulting $t^{\prime}$ integral and find
$$
D_0 \int_0^t W_{\R^2}(t-t^{\prime})F(x,y,t^{\prime})dt^{\prime}= \int_{\R^3} \frac{e^{i\lambda t}-e^{i(\xi^5-\frac{\eta^2}{\xi})t}}{i(\lambda-\xi^5+\frac{\eta^2}{\xi})} e^{i\xi x} \widehat{F}(\xi,\eta,\lambda)d\lambda d\xi d\eta.
$$
We utilize a smooth cut-off function $\phi$ for $[-1,1]$ and its complement $\phi^c=1-\phi$ to break the above integral into three pieces
\begin{multline*}
\mu(t) D_0 \int_0^t W_{\R^2}(t-t^{\prime})F(x,y,t^{\prime})dt^{\prime}\\=\mu(t) \int_{\R^3} \frac{e^{i\lambda t}-e^{i(\xi^5-\frac{\eta^2}{\xi})t}}{i(\lambda-\xi^5+\frac{\eta^2}{\xi})} e^{i\xi x} \phi(\lambda-\xi^5+\tfrac{\eta^2}{\xi})\widehat{F}(\xi,\eta,\lambda)d\lambda d\xi d\eta\\
+\mu(t) \int_{\R^3} \frac{e^{i\lambda t}}{i(\lambda-\xi^5+\frac{\eta^2}{\xi})} e^{i\xi x} \phi^c(\lambda-\xi^5+\tfrac{\eta^2}{\xi}) \widehat{F}(\xi,\eta,\lambda)d\lambda d\xi d\eta\\
-\mu(t) \int_{\R^3} \frac{e^{i(\xi^5-\frac{\eta^2}{\xi})t}}{i(\lambda-\xi^5+\frac{\eta^2}{\xi})} e^{i\xi x} \phi^c(\lambda-\xi^5+\tfrac{\eta^2}{\xi}) \widehat{F}(\xi,\eta,\lambda)d\lambda d\xi d\eta\\ =:G_1(x,t)+G_2(x,t)+G_3(x,t).
\end{multline*}
By Taylor series expansion, we calculate
$$
\frac{e^{i\lambda t}-e^{i(\xi^5-\frac{\eta^2}{\xi})t}}{i(\lambda-\xi^5+\frac{\eta^2}{\xi})}=ie^{i\lambda t} \sum_{k\geq 1} \frac{(-it)^k}{k!} (\lambda-\xi^5+\tfrac{\eta^2}{\xi})^{k-1},
$$
and substitute this series in $G_1$ to obtain
$$
G_1(x,t)= i\mu(t) \int_{\R^3} \sum_{k\geq 1} \frac{(-it)^k}{k!} (\lambda-\xi^5+\tfrac{\eta^2}{\xi})^{k-1}  \phi(\lambda-\xi^5+\tfrac{\eta^2}{\xi}) e^{i\lambda t} e^{i\xi x}\widehat{F}(\xi,\eta,\lambda)d\lambda d\xi d\eta.
$$
We now pass to the Fourier side in both $x$ and $t$ variables
$$
\widehat{G_1}(\xi,\tau)= \int_{\R^2} \sum_{k\geq 1} \frac{(\lambda-\xi^5+\tfrac{\eta^2}{\xi})^{k-1}}{k!} \phi(\lambda-\xi^5+\tfrac{\eta^2}{\xi}) \partial_{\tau}^k \widehat{\mu}(\tau-\lambda) \widehat{F}(\xi,\eta,\lambda)d\lambda d\eta.
$$
Note that 
$$|\tau^M\partial_{\tau}^k \widehat{\mu}(\tau )|=|\widehat{\partial_t^Mt^k \mu(t) }(\tau)|\leq \|\partial_t^Mt^k \mu(t) \|_{L^1}\leq C_M |k|^M, $$
where $M\gg 1$ is fixed. Therefore
\begin{multline*}
\|G_1\|_{\cH^s_{x,t}}=\Big\|\la \xi^2+\eta^2\ra^{\frac{s}{2}}\,\frac{\eta}{\xi}\widehat{G_1}(\xi,\xi^5\pm \tfrac{\eta^2}{\xi})\Big\|_{L^2_{\xi,\eta}}\\
\lesssim \sum_{k=1}^\infty\frac{|k|^{M}\|\lambda^{k-1}\phi(\lambda)\|_{L^\infty}}{k!} \Big\|\la \xi^2+\eta^2\ra^{\frac{s}{2}}\,\frac{\eta}{\xi}\int_{\R^2} \frac{\chi_{|\lambda-\xi^5+\tfrac{\theta^2}{\xi}|<1}}{\la \xi^5\pm \tfrac{\eta^2}{\xi}-\lambda\ra^{M}}|\widehat{F}(\xi,\theta,\lambda) |d\lambda d\theta\Big\|_{L^2_{\xi,\eta}}\\
\les \Big\|\la \xi^2+\eta^2\ra^{\frac{s}{2}}\,\frac{\eta}{\xi}\int_{\R^2} \frac{\chi_{|\lambda-\xi^5+\tfrac{\theta^2}{\xi}|<1}}{\la \tfrac{\theta^2\pm  \eta^2}{\xi}  \ra^{M}}|\widehat{F}(\xi,\theta,\lambda) |d\lambda d\theta\Big\|_{L^2_{\xi,\eta}}:=I_1.
\end{multline*}
We set $\widehat{\psi}(\xi,\theta,\lambda):=\la \xi^2+\theta^2 \ra^{s/2} \la \lambda-\xi^5+\tfrac{\theta^2}{\xi}\ra^{-b} \widehat{F}(\xi,\theta,\lambda) \in L^2_{\xi,\theta,\lambda}$ (here and throughout this proof) and define the kernel
$$
K^1_{\xi}(\eta, (\theta,\lambda)):=\frac{\la \xi^2+\eta^2\ra^{\frac{s}{2}}}{\la \xi^2+\theta^2 \ra^{\frac{s}{2}}}\frac{\eta}{\xi}\frac{1}{\la \tfrac{\theta^2 \pm \eta^2}{\xi} \ra^{M}}\chi_{|\lambda-\xi^5+\tfrac{\theta^2}{\xi}|\lesssim 1},
$$
and consequently arrive at the inequality $I_1\lesssim \|T_1\widehat{\psi}\|_{L^2_{\xi,\eta}}$ where
$$
T_1\widehat{\psi}(\xi,\eta):=\int_{\R^2} K^1_{\xi}(\eta, (\theta,\lambda))\widehat{\psi}(\xi,\theta,\lambda)d\theta d\lambda.
$$
Note that the $\la \lambda-\xi^5+\tfrac{\theta^2}{\xi}\ra^{b}$ multiplier in $K^1_{\xi}$ is ignored as it is $\approx 1$ on the support of $\phi$.

We will prove $\|T_1\widehat{\psi}\|_{L^2_{\xi,\eta}}\lesssim \|\widehat{\psi}\|_{L^2_{\xi,\theta,\lambda}}$ by showing 
$$
\int_{\R} |K^1_{\xi}(\eta, (\theta,\lambda))| d\eta \lesssim 1\quad\text{and}\quad \int_{\R^2} |K^1_{\xi}(\eta, (\theta,\lambda))| d\theta d\lambda \lesssim 1
$$
uniformly in $\xi$ using Lemma~\ref{lem:schur}.  We do this in two regions.

\noindent
\textsc{Case 1:} $|\eta| \lesssim |\theta|$.

In this region $\la \xi^2+\eta^2\ra \lesssim \la \xi^2+\theta^2 \ra$ is valid. Using this and the change of variable $\rho=\frac{\eta^2}{\xi}$ we get
$$
\int\limits_{|\eta| \lesssim |\theta|} |K^1_{\xi}(\eta, (\theta,\lambda))| d\eta \lesssim \int_{\R} \frac{d\rho}{\la \frac{\theta^2}{\xi}\pm \rho \ra^M}\lesssim 1.
$$ 
For the other integral, we have
$$
\int\limits_{|\theta| \gtrsim |\eta|} \int_{\R} |K^1_{\xi}(\eta, (\theta,\lambda))| d\lambda d\theta \lesssim \int\limits_{|\theta| \gtrsim |\eta|} \int\limits_{|\lambda-\xi^5+\tfrac{\theta^2}{\xi}|\lesssim 1} \frac{|\theta|}{|\xi|} \frac{1}{\la \tfrac{\theta^2 \pm \eta^2}{\xi} \ra^{M}} d\lambda d\theta \lesssim \int_{\R} \frac{d\rho}{\la \rho \pm \frac{\eta^2}{\xi} \ra^M}\lesssim 1,
$$
where now $\rho=\frac{\theta^2}{\xi}$.

\noindent
\textsc{Case 2:} $|\eta| \gg |\theta|$.

In this case we observe $\la \tfrac{\theta^2 \pm \eta^2}{\xi} \ra \approx \la \tfrac{\eta^2}{\xi} \ra$. Then implement the change of variable $\rho=\frac{\eta^2}{\xi}$ and hence obtain
$$
\int\limits_{|\eta| \gg |\theta|} |K^1_{\xi}(\eta, (\theta,\lambda))| d\eta \lesssim \int\limits_{|\rho| \gg \frac{\theta^2}{|\xi|}} \frac{\la \xi^2+\xi\rho\ra^{\frac{s}{2}}}{\la \xi^2+\theta^2 \ra^{\frac{s}{2}}} \frac{d\rho}{\la \rho \ra^M} \lesssim \int_{\R} \frac{\la \xi \ra^s + \la \xi \ra^{\frac{s}{2}} \la \rho \ra^{\frac{s}{2}}}{\la \xi \ra^s \la \rho \ra^M} d\rho \lesssim 1.
$$
We handle the other integral by considering the cases $|\xi|<1$ and $|\xi|>1$ seperately. 
In the latter case, we have  (for $0\leq s<1$)
\begin{multline}\label{eq:K11}
\int\limits_{|\theta| \ll |\eta|} \int_{\R} |K^1_{\xi}(\eta, (\theta,\lambda))| d\lambda d\theta \lesssim \int\limits_{|\theta| \ll |\eta|} \frac{( \xi^2+\eta^2)^{\frac{s}{2}}}{( \xi^2+\theta^2 )^{\frac{s}{2}}}\frac{|\eta|}{|\xi|}\frac{d\theta}{\la \frac{\eta^2}{\xi}\ra^M}\\
  =\int\limits_{|\rho| \ll |\beta| }\frac{(1+\beta^2)^{\frac{s}{2}}}{(1+\rho^2)^{\frac{s}{2}}} |\beta\xi| \frac{d\rho}{\la \beta^2 \xi \ra^M}  \les \frac { \beta^2|\xi| }{\la \beta^2 \xi \ra^M} \les 1.
\end{multline}
In the former case, we have (for $0\leq s<1$)
\be\label{eq:K12}
\int\limits_{|\theta| \ll |\eta|} \int_{\R} |K^1_{\xi}(\eta, (\theta,\lambda))| d\lambda d\theta \lesssim \int\limits_{|\theta| \ll |\eta|} \frac{\la \eta\ra^s}{\la\theta\ra^s  }\frac{|\eta|}{|\xi|}\frac{d\theta}{\la \frac{\eta^2}{\xi}\ra^M}\les \frac{\frac{\eta^2}{|\xi|} }{\la \frac{\eta^2}{\xi}\ra^M}\les 1.
\ee
This proves $\|G_1\|_{\cH^s_{x,t}}\lesssim \|\widehat{\psi}\|_{L^2_{\xi,\theta,\lambda}}=\|F\|_{X^{s,-b}}$.
A similar argument extends this to $s\geq 1$ provided that we choose $M$ suitably large.

We now pass to proving $\|G_2\|_{\cH^s_{x,t}}\lesssim \|F\|_{X^{s,-b}}$. From above we have 
$$
\widehat{G_2}(\xi,\eta)\approx \int_{\R^2} \frac{\widehat{\mu}(\eta-\lambda) \widehat{F}(\xi,\theta,\lambda)}{i(\lambda-\xi^5+\frac{\theta^2}{\xi})}\phi^c(\lambda-\xi^5+\tfrac{\theta^2}{\xi})d\lambda d\theta,
$$ 
and using $|\lambda-\xi^5+\frac{\theta^2}{\xi}|\geq 1$ on the support of $\phi^c(\lambda-\xi^5+\frac{\theta^2}{\xi})$ we deduce
$$
|\widehat{G_2}(\xi,\eta)|\lesssim \int_{\R^2} \frac{|\widehat{\mu}(\eta-\lambda)|  |\widehat{F}(\xi,\theta,\lambda)|}{\la \lambda-\xi^5+\frac{\theta^2}{\xi}\ra} d\lambda d\theta.
$$
To prove the assertion we write
\begin{multline*}
\|G_2\|_{\cH^s_{x,t}}\lesssim \Big\|\la \xi^2+\eta^2\ra^{\frac{s}{2}}\,\frac{|\eta|}{|\xi|}\widehat{G_2}(\xi,\xi^5\pm \tfrac{\eta^2}{\xi})\Big\|_{L^2_{\xi,\eta}}\\
\lesssim \Big\|\frac{|\eta|\la \xi^2+\eta^2\ra^{\frac{s}{2}}}{|\xi|}\int_{\R^2} \frac{ |\widehat{F}(\xi,\theta,\lambda)| }{\la \lambda-\xi^5+\frac{\theta^2}{\xi}\ra \la \xi^5\pm \tfrac{\eta^2}{\xi}-\lambda\ra^M}\, d\lambda d\theta\Big\|_{L^2_{\xi,\eta}},
\end{multline*}
where we used the usual bound for $\widehat{\mu}$. 
We now concentrate on the case $0\leq s\leq \frac12$. Defining $\psi$ as above, we have 
$$
\|G_2\|_{\cH^s_{x,t}}\les 
\Big\|\int_{\R^2} \frac{\la \xi^2+\eta^2\ra^{\frac{s}{2}}}{\la \xi^2+\theta^2\ra^{\frac{s}{2}}}\,\frac{|\eta|}{|\xi|} \frac{\la \lambda-\xi^5+\frac{\theta^2}{\xi}\ra^{b-1}}{\la \xi^5\pm \tfrac{\eta^2}{\xi}-\lambda\ra^M} |\widehat{\psi}(\xi,\theta,\lambda)| d\lambda d\theta\Big\|_{L^2_{\xi,\eta}}:=I_2,
$$
We therefore work with the kernel
$$
K^2_{\xi}(\eta,(\theta,\lambda)):=\frac{\la \xi^2+\eta^2\ra^{\frac{s}{2}}}{\la \xi^2+\theta^2\ra^{\frac{s}{2}}}\,\frac{|\eta|}{|\xi|} \frac{\la \lambda-\xi^5+\frac{\theta^2}{\xi}\ra^{b-1}}{\la \xi^5\pm \tfrac{\eta^2}{\xi}-\lambda\ra^M}
$$
as above.

\noindent
\textsc{Case 1:} $|\eta|\lesssim |\theta|$ or $|\eta|\les \la \xi\ra$.

We apply Lemma~\ref{lem:schur} with $q(\eta)=1$ and $p(\theta,\lambda)=\la \lambda-\xi^5+\tfrac{\theta^2}{\xi}\ra^{b-1}$ in this region. By $\la \xi^2+\eta^2\ra \lesssim \la \xi^2+\theta^2 \ra$ we have
\begin{multline*} 
\int\limits_{|\eta| \lesssim |\theta|} |K^2_{\xi}(\eta, (\theta,\lambda))| d\eta \lesssim \int\limits_{|\eta| \lesssim |\theta|} \frac{|\eta|}{|\xi|} \frac{\la \lambda-\xi^5+\frac{\theta^2}{\xi}\ra^{b-1}}{\la \xi^5\pm \tfrac{\eta^2}{\xi}-\lambda\ra^M} d\eta \\ \lesssim \int_{\R} \frac{\la \lambda-\xi^5+\frac{\theta^2}{\xi}\ra^{b-1}}{\la \xi^5\pm \rho-\lambda\ra^M} d\rho \lesssim \la \lambda-\xi^5+\tfrac{\theta^2}{\xi}\ra^{b-1}=p(\theta,\lambda),
\end{multline*}
where $\rho=\frac{\eta^2}{\xi}$. We also have
\begin{multline*} 
\int\limits_{|\theta| \gtrsim |\eta|} \int_{\R} |K^2_{\xi}(\eta, (\theta,\lambda))| p(\theta,\lambda) d\lambda d\theta \lesssim \int\limits_{|\theta| \gtrsim |\eta|} \int_{\R} \frac{|\eta|}{|\xi|} \frac{\la \lambda-\xi^5+\frac{\theta^2}{\xi}\ra^{2b-2}}{\la \xi^5\pm \tfrac{\eta^2}{\xi}-\lambda\ra^M}  d\lambda d\theta \\ \lesssim \int_{\R} \frac{|\theta|}{|\xi|} \frac{d\theta}{\la \frac{\theta^2\mp \eta^2}{\xi}\ra^{2-2b}} \lesssim \int_{\R} \frac{d\rho}{\la \rho \mp \frac{\eta^2}{\xi}\ra^{2-2b}}\lesssim 1,
\end{multline*}
where $\rho=\frac{\theta^2}{\xi}$ and provided that $b<\frac{1}{2}$.

\noindent
\textsc{Case 2:} $|\eta|\gg |\theta|$ and $|\eta|\gg \la \xi\ra$.

For this region $q(\eta)=| \eta |^{-s}$ and $p(\theta,\lambda)=\la \lambda-\xi^5+\tfrac{\theta^2}{\xi}\ra^{b-1} \la \xi^2+\theta^2\ra^{-\frac{s}{2}}$ are the suitable   functions for Lemma~\ref{lem:schur}. Upon this setting we get
\begin{multline*}
\int\limits_{|\eta| \gg |\theta|} |K^2_{\xi}(\eta, (\theta,\lambda))| q(\eta) d\eta \lesssim \int\limits_{|\eta| \gg |\theta|} \frac{|\eta|}{|\xi|} \frac{1}{\la \xi^2+\theta^2\ra^{\frac{s}{2}}} \frac{\la \lambda-\xi^5+\frac{\theta^2}{\xi}\ra^{b-1}}{\la \xi^5\pm \tfrac{\eta^2}{\xi}-\lambda\ra^M} d\eta \\ \lesssim \int_{\R}  \frac{\la \lambda-\xi^5+\frac{\theta^2}{\xi}\ra^{b-1}}{\la \xi^2+\theta^2\ra^{\frac{s}{2}}} \frac{d\rho}{\la \xi^5\pm \rho-\lambda\ra^M}
\lesssim   \frac{\la \lambda-\xi^5+\frac{\theta^2}{\xi}\ra^{b-1}}{\la \xi^2+\theta^2\ra^{\frac{s}{2}}}=p(\theta,\lambda),
\end{multline*}
where we implemented the change of variable $\rho=\frac{\eta^2}{\xi}$. For the $(\lambda,\theta)$ integral we obtain
\begin{multline*}
\int\limits_{|\theta| \ll |\eta|} \int_{\R} |K^2_{\xi}(\eta, (\theta,\lambda))| p(\theta,\lambda) d\lambda d\theta \lesssim \int\limits_{|\theta| \ll |\eta|} \int_{\R} \frac{1}{\la \xi^2+\theta^2\ra^s} \frac{|\eta|^{1+s}}{|\xi|} \frac{\la \lambda-\xi^5+\frac{\theta^2}{\xi}\ra^{2b-2}}{\la \xi^5\pm \tfrac{\eta^2}{\xi}-\lambda\ra^M} d\lambda d\theta\\
\lesssim \int\limits_{|\theta| \ll |\eta|} \frac{1}{\la \xi^2+\theta^2\ra^s} \frac{|\eta|^{1+s}}{|\xi|} \frac{d\theta}{\la \frac{\theta^2\mp \eta^2}{\xi}\ra^{2-2b}} \lesssim | \eta |^{-s} \int\limits_{|\theta| \ll |\eta|} \frac{1}{\la \xi^2+\theta^2\ra^s} \frac{|\eta|^{1+2s}}{|\xi|} \frac{d\theta}{\la \frac{\eta^2}{\xi}\ra^{2-2b}}\\ \lesssim |  \eta |^{-s}=q(\eta).
\end{multline*}
The last inequality follows by considering the cases $|\xi|<1$ and $|\xi|>1$ as in \eqref{eq:K11} and \eqref{eq:K12} provided that   $0\leq s\leq \frac12$ and $2-2b>1$. 

For $s>\frac12$, the proof is the same for case 1. It remains to consider the case  when $|\eta|\gg|\theta|$ and $|\eta|\gg \la \xi\ra $, which will contribute the second summand on the right hand side of \eqref{weird}. We estimate the contribution of this region to $\|G_2\|_{\cH^s_{x,t}}$ by 
$$
  \Big\|\frac{|\eta|^{1+s}}{|\xi|}\int_{|\theta|\ll|\eta|} \frac{ |\widehat{F}(\xi,\theta,\lambda)| }{\la \lambda-\xi^5+\frac{\theta^2}{\xi}\ra \la \xi^5\pm \tfrac{\eta^2}{\xi}-\lambda\ra^M}\, d\lambda d\theta\Big\|_{L^2_{\xi,\eta}},
$$
It suffices to prove that the operator with the kernel 
$$
\widetilde{K}^2_\xi(\eta,(\theta,\lambda)):=\frac{|\eta|^{1+s}}{|\xi|^{\frac{s}2+\frac34}\la \theta^2+\xi^2\ra^{\frac14+}} \frac{ \chi_{|\theta|\ll|\eta|} }{\la \lambda-\xi^5+\frac{\theta^2}{\xi}\ra^{1-b_1} \la \xi^5\pm \tfrac{\eta^2}{\xi}-\lambda\ra^M}
$$
is bounded from $L^2_\eta$ to $L^2_{\theta,\lambda}$ uniformly in $\xi$. Let $q(\eta)=|\eta|^{-s}$ and 
$$p(\theta,\lambda)=\frac{1}{|\xi|^{\frac{s}2-\frac14}\la \theta^2+\xi^2\ra^{\frac14+} \la \lambda-\xi^5+\frac{\theta^2}{\xi}\ra^{1-b_1} }.$$
It is easy to see that $\int \widetilde{K}^2_\xi(\eta,(\theta,\lambda)) q(\eta) d\eta \les p(\theta,\lambda) $ by change of variable $\rho=\frac{\eta^2}{\xi}.$ 
We also have 
\begin{multline*}
\frac1{q(\eta)} \int \widetilde{K}^2_\xi(\eta,(\theta,\lambda)) p(\theta,\lambda)  d\theta d\lambda  
\\ =   \int  \frac{|\eta|^{1+2s}}{|\xi|^{s+\frac12}\la \theta^2+\xi^2\ra^{\frac12+}} \frac{ \chi_{|\theta|\ll|\eta|} }{\la \lambda-\xi^5+\frac{\theta^2}{\xi}\ra^{2-2b_1} \la \xi^5\pm \tfrac{\eta^2}{\xi}-\lambda\ra^M} d\theta d\lambda \\
\les \int  \frac{|\eta|^{1+2s}}{|\xi|^{s+\frac12}\la \theta^2+\xi^2\ra^{\frac12+}  (\frac{\eta^2}{\xi})^{2-2b_1}  } d\theta \les \Big(\frac{\eta^2}{\xi}\Big)^{s+2b_1-\frac32}\les 1
\end{multline*}
provided that $b_1\leq\frac34-\frac{s}2$.
 
For $G_3$ we note that
$$
\widehat{G_3}(\xi,\eta)=\int_{\R^2} \frac{\widehat{\mu}(\eta-\xi^5+\frac{\theta^2}{\xi}) \widehat{F}(\xi,\theta,\lambda)}{i(\lambda-\xi^5+\frac{\theta^2}{\xi})}\phi^c(\lambda-\xi^5+\tfrac{\theta^2}{\xi})d\lambda d\theta,
$$ 
and hence as in $G_2$ we obtain
$$
|\widehat{G_3}(\xi,\xi^5\pm \tfrac{\eta^2}{\xi})|\lesssim \int_{\R^2} \frac{|\widehat{\mu}(\frac{\theta^2\pm \eta^2}{\xi})|  |\widehat{F}(\xi,\theta,\lambda)|}{\la \lambda-\xi^5+\frac{\theta^2}{\xi} \ra} d\lambda d\theta \lesssim \int_{\R^2} \frac{\la \lambda-\xi^5+\frac{\theta^2}{\xi} \ra^{b-1}}{\la \frac{\theta^2\pm \eta^2}{\xi} \ra^M}  \frac{|\widehat{\psi}(\xi,\theta,\lambda)|}{\la \xi^2+\theta^2 \ra^{\frac{s}{2}}} d\lambda d\theta.
$$
We then establish 
$$
\|G_3\|_{\cH^s_{x,t}}\lesssim \Big\|\int_{\R^2} \frac{\la \xi^2+\eta^2\ra^{\frac{s}{2}}}{\la \xi^2+\theta^2\ra^{\frac{s}{2}}}\,\frac{|\eta|}{|\xi|} \frac{\la \lambda-\xi^5+\frac{\theta^2}{\xi}\ra^{b-1}}{\la \frac{\theta^2\pm \eta^2}{\xi} \ra^M} |\widehat{\psi}(\xi,\theta,\lambda)|d\lambda d\theta\Big\|_{L^2_{\xi,\eta}}:=\|T_3\widehat{\psi}\|_{L^2_{\xi,\eta}},
$$
where $T_3$ is an operator on $L^2$ with the kernel
$$
K^3_{\xi}(\eta,(\theta,\lambda)):=\frac{\la \xi^2+\eta^2\ra^{\frac{s}{2}}}{\la \xi^2+\theta^2\ra^{\frac{s}{2}}}\,\frac{|\eta|}{|\xi|} \frac{\la \lambda-\xi^5+\frac{\theta^2}{\xi}\ra^{b-1}}{\la \frac{\theta^2\pm \eta^2}{\xi}\ra^M}.
$$

\noindent
\textsc{Case 1:} $|\eta|\lesssim |\theta|$.

We  use Lemma~\ref{lem:schur} with the functions $q(\eta)=1$ and $p(\theta,\lambda)=\la \lambda-\xi^5+\tfrac{\theta^2}{\xi}\ra^{b-1}$ in this region. By $\la \xi^2+\eta^2\ra \lesssim \la \xi^2+\theta^2 \ra$ we have
\begin{multline*} 
\int\limits_{|\eta| \lesssim |\theta|} |K^3_{\xi}(\eta, (\theta,\lambda))| d\eta \lesssim \int\limits_{|\eta| \lesssim |\theta|} \frac{|\eta|}{|\xi|} \frac{\la \lambda-\xi^5+\frac{\theta^2}{\xi}\ra^{b-1}}{\la \frac{\theta^2\pm \eta^2}{\xi}\ra^M} d\eta \\ \lesssim \int_{\R} \frac{\la \lambda-\xi^5+\frac{\theta^2}{\xi}\ra^{b-1}}{\la  \rho\pm\frac{\theta^2}{\xi}\ra^M} d\rho \lesssim \la \lambda-\xi^5+\tfrac{\theta^2}{\xi}\ra^{b-1}=p(\theta,\lambda),
\end{multline*}
where $\rho=\frac{\eta^2}{\xi}$. We also have
\begin{multline*} 
\int\limits_{|\theta| \gtrsim |\eta|} \int_{\R} |K^3_{\xi}(\eta, (\theta,\lambda))| p(\theta,\lambda) d\lambda d\theta \lesssim \int\limits_{|\theta| \gtrsim |\eta|} \int_{\R} \frac{|\eta|}{|\xi|} \frac{\la \lambda-\xi^5+\frac{\theta^2}{\xi}\ra^{2b-2}}{\la \frac{\theta^2\pm \eta^2}{\xi}\ra^M}  d\lambda d\theta \\ \lesssim \int_{\R} \frac{|\theta|}{|\xi|} \frac{d\theta}{\la \frac{\theta^2\mp \eta^2}{\xi}\ra^{M}} \lesssim \int_{\R} \frac{d\rho}{\la \rho \mp \frac{\eta^2}{\xi}\ra^{M}}\lesssim 1,
\end{multline*}
where $\rho=\frac{\theta^2}{\xi}$ and provided that $b<\frac{1}{2}$.

\noindent
\textsc{Case 2:} $|\eta|\gg |\theta|$.

For this region $q(\eta)=\la \xi^2+\eta^2\ra^{-\frac{s}{2}}$ and $p(\theta,\lambda)=\la \lambda-\xi^5+\tfrac{\theta^2}{\xi}\ra^{b-1} \la \xi^2+\theta^2\ra^{-\frac{s}{2}}$ are the suitable Schur functions. Upon this setting we get
\begin{multline*}
\int\limits_{|\eta| \gg |\theta|} |K^3_{\xi}(\eta, (\theta,\lambda))| q(\eta) d\eta \lesssim \int\limits_{|\eta| \gg |\theta|} \frac{|\eta|}{|\xi|} \frac{1}{\la \xi^2+\theta^2\ra^{\frac{s}{2}}} \frac{\la \lambda-\xi^5+\frac{\theta^2}{\xi}\ra^{b-1}}{\la \frac{\theta^2\pm \eta^2}{\xi}\ra^M} d\eta \\ \lesssim \int_{\R}  \frac{\la \lambda-\xi^5+\frac{\theta^2}{\xi}\ra^{b-1}}{\la \xi^2+\theta^2\ra^{\frac{s}{2}}} \frac{d\rho}{\la   \rho\pm\frac{\theta^2}{\xi} \ra^M}
\lesssim   \frac{\la \lambda-\xi^5+\frac{\theta^2}{\xi}\ra^{b-1}}{\la \xi^2+\theta^2\ra^{\frac{s}{2}}}=p(\theta,\lambda),
\end{multline*}
where we implemented the change of variable $\rho=\frac{\eta^2}{\xi}$. For the $(\lambda,\theta)$ integral we obtain
\begin{multline*}
\int\limits_{|\theta| \ll |\eta|} \int_{\R} |K^3_{\xi}(\eta, (\theta,\lambda))| p(\theta,\lambda) d\lambda d\theta \lesssim \int\limits_{|\theta| \ll |\eta|} \int_{\R} \frac{\la \xi^2+\eta^2\ra^{\frac{s}{2}}}{\la \xi^2+\theta^2\ra^s} \frac{|\eta|}{|\xi|} \frac{\la \lambda-\xi^5+\frac{\theta^2}{\xi}\ra^{2b-2}}{\la \frac{\theta^2\pm \eta^2}{\xi}\ra^M} d\lambda d\theta\\
\lesssim \int\limits_{|\theta| \ll |\eta|} \frac{\la \xi^2+\eta^2\ra^{\frac{s}{2}}}{\la \xi^2+\theta^2\ra^s} \frac{|\eta|}{|\xi|} \frac{d\theta}{\la \frac{\theta^2\mp \eta^2}{\xi}\ra^{M}} \lesssim \la \xi^2+\eta^2\ra^{-\frac{s}{2}} \int\limits_{|\theta| \ll |\eta|} \frac{\la \xi^2+\eta^2\ra^s}{\la \xi^2+\theta^2\ra^s} \frac{|\eta|}{|\xi|} \frac{d\theta}{\la \frac{\eta^2}{\xi}\ra^{M}}\\ \lesssim \la \xi^2+\eta^2\ra^{-\frac{s}{2}}=q(\eta).
\end{multline*}
The last inequality follows by considering the cases $|\xi|<1$ and $|\xi|>1$ as in \eqref{eq:K11} and \eqref{eq:K12} provided that   $0\leq s $ and $M$ is sufficiently large. 
\end{proof}

\section{Nonlinear Estimates}\label{sec:nonlin}  We now establish estimates for the nonlinear term  in \eqref{eq:duhamel}.   
\begin{theorem}\label{thm:nonlin1}
Fix $s>0$ and $ a<\min(\frac{2s}3,\frac13)$. If $b<\tfrac12$ is sufficiently close to $\tfrac12$, then
$$
\|(u^2)_x\|_{X^{s+a,  -b }}\les \|u\|_{X^{s,b}}^2.
$$
\end{theorem}
\begin{proof}
By duality, see e.g. \cite{etnls}, it suffices to prove that  
\be\label{nonlinmainineq}
\int_{\R^6} \frac{|\xi| \la\xi,\theta\ra^{s+a} |f(\sigma)f_1(\sigma_1)f_2(\sigma -\sigma_1)|}{\la\xi_1,\theta_1\ra^s \la \xi-\xi_1,\theta-\theta_1\ra^s \la \tau\ra^b\la\tau_1\ra^b \la \tau_2\ra^b} d\sigma d\sigma_1\les \|f\|_{L^2}\|f_1\|_{L^2}\|f_2\|_{L^2},
\ee
where $\sigma=(\xi,\theta,\lambda),$ $d\sigma=d\xi d\theta d\lambda$,  similarly for $\sigma_1$ and $d\sigma_1$. Moreover
  $\tau:=\lambda-\xi^5+\frac{\theta^2}{\xi}$,  similarly for $ \tau_1$,  and $\tau_2=\lambda-\lambda_1-(\xi-\xi_1)^5+\frac{(\theta-\theta_1)^2}{\xi-\xi_1}$.

We have the identity 
$$
\tau-\tau_1-\tau_2= \xi_1^5+(\xi-\xi_1)^5-\xi^5+\frac{\theta^2}{\xi}  -\frac{\theta_1^2}{\xi_1} -\frac{(\theta-\theta_1)^2}{\xi-\xi_1}
$$
 $$=
-5\xi\xi_1(\xi-\xi_1)(\xi^2-\xi\xi_1+\xi_1^2) -\frac{(\theta\xi_1-\theta_1\xi)^2}{\xi\xi_1(\xi-\xi_1)}.
$$
Noting that both summands have the same sign, we have 
$$
|\tau-\tau_1-\tau_2|\gtrsim |\xi\xi_1(\xi-\xi_1)|  (\xi^2+\xi_1^2) + \frac{(\theta\xi_1-\theta_1\xi)^2}{|\xi\xi_1(\xi-\xi_1)|}=:M.
$$
We will prove the inequality \eqref{nonlinmainineq} by considering various regions for the parameters involved. In the integral signs we will omit the domain of integration since it will be clear from the context.  

We first consider the nonresonant region $|\xi-\xi_1|,|\xi_1|\gtrsim 1$. By symmetry, it suffices to  consider the following  cases: $\la\tau\ra\gtrsim \la M\ra$ and $\la \tau_1\ra\gtrsim \la M\ra$.

\noindent
\textsc{Case 1:} $\la\tau\ra\gtrsim \la M\ra$. 

By the Cauchy--Schwarz inequality and the convolution structure it suffices to obtain the bound below, see e.g. \cite{etnls} 
$$
I:=\sup_{\xi,\theta,\lambda} \int  \frac{\xi^2 \la\xi,\theta\ra^{2s+2a} d\xi_1 d\theta_1 d\lambda_1 }{\la\xi_1,\theta_1\ra^{2s} \la \xi-\xi_1,\theta-\theta_1\ra^{2s} \la M\ra^{1-} \la\tau_1\ra^{1-} \la \tau_2\ra^{1-} }   \les 1.
$$
Evaluating the $\lambda_1$ integral using Lemma~\ref{lem:sums}, we have
 $$
 I\les \sup_{\xi,\theta,\lambda} \int \frac{\xi^2 \la\xi,\theta\ra^{2s+2a} d\xi_1 d\theta_1  }{\la\xi_1,\theta_1\ra^{2s} \la \xi-\xi_1,\theta-\theta_1\ra^{2s} \la M\ra^{1-} \la\lambda -\xi_1^5-(\xi-\xi_1)^5+\frac{\theta_1^2}{\xi_1}+\frac{(\theta-\theta_1)^2}{\xi-\xi_1}\ra^{1-} }  
 $$
 In the $\theta_1$ integral, we let  
\be\label{etatheta1}\left\{
 \begin{array}{lll}
&\eta=\frac{\theta_1^2}{\xi_1}+\frac{(\theta-\theta_1)^2}{\xi-\xi_1}-\frac{\theta^2}{\xi}=  \frac{(\theta\xi_1-\theta_1\xi)^2}{\xi\xi_1(\xi-\xi_1)},\\
& d\eta=2\frac{|\xi\theta_1-\xi_1\theta|}{|\xi_1(\xi-\xi_1)|} d\theta_1
 =2\frac{|\eta|^{\frac12} |\xi|^{\frac12}    }{|\xi_1|^{\frac12}|\xi-\xi_1|^{\frac12}} d\theta_1,\\
& M=|\eta |+|\xi\xi_1(\xi-\xi_1)|  (\xi^2+\xi_1^2) 
\end{array}\right. \ee
 to obtain the integral  
 $$  \int  \frac{|\xi|^{\frac32} \la\xi,\theta\ra^{2s+2a} |\xi_1|^{\frac12}|\xi-\xi_1|^{\frac12} d\xi_1 d\eta   }{\la\xi_1,\theta_1\ra^{2s} \la \xi-\xi_1,\theta-\theta_1\ra^{2s} \la M\ra^{1-} |\eta|^{\frac12} \la\lambda+\frac{\theta^2}{\xi} +\eta-\xi_1^5-(\xi-\xi_1)^5 \ra^{1-} },
 $$
 where $\theta_1=\frac{\theta\xi_1}{\xi}\pm\sqrt{\tfrac{|\eta\xi_1(\xi-\xi_1)|}{|\xi|}}$.

In the case $|\theta|\les \la \xi\ra$, we bound this by (using Lemma~\ref{lem:sums} twice)
\begin{multline*}
\int  \frac{|\xi|^{\frac32} \la\xi \ra^{2s+2a} |\xi_1|^{\frac12}|\xi-\xi_1|^{\frac12} d\xi_1 d\eta   }{\la\xi_1 \ra^{2s} \la \xi-\xi_1 \ra^{2s} \la M\ra^{1-} |\eta|^{\frac12} \la\lambda+\frac{\theta^2}{\xi} +\eta-\xi_1^5-(\xi-\xi_1)^5 \ra^{1-} }\\
\les \int  \frac{|\xi|^{\frac32} \la\xi \ra^{2s+2a} |\xi_1|^{\frac12}|\xi-\xi_1|^{\frac12} d\xi_1   }{\la\xi_1 \ra^{2s} \la \xi-\xi_1 \ra^{2s} \la  \xi\xi_1(\xi-\xi_1)   (\xi^2+\xi_1^2) \ra^{1-}   }\\
\les \int  \frac{|\xi|^{\frac12+} \la\xi \ra^{2s+2a}  d\xi_1   }{\la\xi_1 \ra^{2s+\frac12-} \la \xi-\xi_1 \ra^{2s+\frac12-}    (|\xi|+1)^{2-}   }
\end{multline*}
which is bounded in $\xi$ provided that $a<\min(1,s+\frac34)$. 

In the case $|\theta|\gg \la \xi\ra$, by symmetry  we can assume that $|\theta-\theta_1|\gtrsim |\theta|$, and hence we have the bound 
$$ \int  \frac{|\xi|^{\frac32} |\theta|^{ 2a} |\xi_1|^{\frac12}|\xi-\xi_1|^{\frac12} d\xi_1 d\eta   }{\big\la \frac{\theta\xi_1}{\xi}\pm\sqrt{\tfrac{|\eta\xi_1(\xi-\xi_1)|}{|\xi|}}\big\ra^{2s}   \la M\ra^{1-} |\eta|^{\frac12} \la\lambda+\frac{\theta^2}{\xi} +\eta-\xi_1^5-(\xi-\xi_1)^5 \ra^{1-} }.
$$
When $\tfrac{|\eta\xi_1(\xi-\xi_1)|}{|\xi|}<1$,  by integrating in the $\eta$ variable and using $a<\min( s,\frac12)$ we bound the last integral by
$$\les  \int  \frac{|\xi|^{\frac32} |\theta|^{ 2a} |\xi_1|^\frac12 |\xi-\xi_1|^\frac12  d\xi_1   }{\big| \frac{\theta\xi_1}{\xi} \big|^{ 2a }    \la  \xi\xi_1(\xi-\xi_1)   (\xi^2+\xi_1^2) \ra^{1-}   }.$$ 
When $\tfrac{|\eta\xi_1(\xi-\xi_1)|}{|\xi|}>1$ we bound the same integral by 
\begin{multline*} \int  \frac{|\xi| |\theta|^{ 2a} |\xi_1||\xi-\xi_1| d\xi_1 d\eta   }{\big\la \frac{\theta\xi_1}{\xi}\pm\sqrt{\tfrac{|\eta\xi_1(\xi-\xi_1)|}{|\xi|}}\big\ra^{2s}   \la M\ra^{1-} \sqrt{\tfrac{|\eta\xi_1(\xi-\xi_1)|}{|\xi|}} \la\lambda+\frac{\theta^2}{\xi} +\eta-\xi_1^5-(\xi-\xi_1)^5 \ra^{1-} }\\
\les  \int \frac{|\xi| |\theta|^{ 2a} |\xi_1| |\xi-\xi_1|  d\xi_1   }{\big| \frac{\theta\xi_1}{\xi} \big|^{2a}    \la  \xi\xi_1(\xi-\xi_1)   (\xi^2+\xi_1^2) \ra^{1-}   } .
\end{multline*}
Notice that in both cases the integrals are estimated by
$$\int |\xi|^{-1+2a+}   |\xi_1|^{-\frac12-2a+} |\xi-\xi_1|^{-\frac12+}  d\xi_1 \les 1,$$
  provided that $0<a< \min(s,\frac12)$.  

\noindent
\textsc{Case 2:} $\la\tau_1\ra\gtrsim \la M\ra$. 

In this case the needed bound boils down to  
$$
I:=\sup_{\xi_1,\theta_1,\lambda_1} \int  \frac{\xi^2 \la\xi,\theta\ra^{2s+2a} d\xi  d\theta  d\lambda  }{\la\xi_1,\theta_1\ra^{2s} \la \xi-\xi_1,\theta-\theta_1\ra^{2s} \la M\ra^{1-} \la\tau \ra^{1-} \la \tau_2\ra^{1-} }   \les 1.
$$
Evaluating the $\lambda $ integral, we have
 $$
 I\les \sup_{\xi_1,\theta_1,\lambda_1}\int  \frac{\xi^2 \la\xi,\theta\ra^{2s+2a} d\xi  d\theta  }{\la\xi_1,\theta_1\ra^{2s} \la \xi-\xi_1,\theta-\theta_1\ra^{2s} \la M\ra^{1-} \la\lambda_1 -\xi^5+(\xi-\xi_1)^5+\frac{\theta ^2}{\xi }-\frac{(\theta-\theta_1)^2}{\xi-\xi_1}\ra^{1-} }  
 $$
 In the $\theta $ integral, we let  
 \be\label{etatheta}\left\{
 \begin{array}{lll}
&\eta=\frac{\theta_1^2}{\xi_1}+\frac{(\theta-\theta_1)^2}{\xi-\xi_1}-\frac{\theta^2}{\xi}=  \frac{(\theta\xi_1-\theta_1\xi)^2}{\xi\xi_1(\xi-\xi_1)},\\
& d\eta=2\frac{|\xi\theta_1-\xi_1\theta|}{|\xi(\xi-\xi_1)|} d\theta
 =2\frac{|\eta|^{\frac12} |\xi_1|^{\frac12}    }{|\xi|^{\frac12}|\xi-\xi_1|^{\frac12}} d\theta,\\
& M=|\eta |+|\xi\xi_1(\xi-\xi_1)|  (\xi^2+\xi_1^2) 
\end{array}\right. \ee
 and obtain the integral  
 $$  \int  \frac{|\xi|^{\frac52} \la\xi,\theta\ra^{2s+2a} |\xi_1|^{-\frac12}|\xi-\xi_1|^{\frac12} d\xi  d\eta   }{\la\xi_1,\theta_1\ra^{2s} \la \xi-\xi_1,\theta-\theta_1\ra^{2s} \la M\ra^{1-} |\eta|^{\frac12}  \la\lambda_1 -\xi^5+(\xi-\xi_1)^5+\frac{\theta_1^2}{\xi_1 }-\eta\ra^{1-} },
 $$
 where $\theta =\frac{\theta_1\xi}{\xi_1}\pm\sqrt{\tfrac{|\eta\xi (\xi-\xi_1)|}{|\xi_1|}}$.
 
 In the region $|\theta|\les \la \xi\ra$, we bound the integral by 
\begin{multline*}
  \int \frac{|\xi|^{\frac52} \la\xi \ra^{2s+2a} |\xi_1|^{-\frac12}|\xi-\xi_1|^{\frac12} d\xi    }{\la\xi_1 \ra^{2s} \la \xi-\xi_1 \ra^{2s} \la \xi\xi_1(\xi-\xi_1)   (\xi^2+\xi_1^2) \ra^{1-} \la\lambda_1 -\xi^5+(\xi-\xi_1)^5+\frac{\theta_1^2}{\xi_1 } \ra^{\frac12-} } \\ \les   \int  \frac{|\xi|^{\frac32+} \la\xi \ra^{ 2a-\frac52+}  d\xi    }{|\xi_1 |  \la\lambda_1 -\xi^5+(\xi-\xi_1)^5+\frac{\theta_1^2}{\xi_1 } \ra^{\frac12-} }.
 \end{multline*}
In the region $|\xi|\les |\xi_1|$, this is bounded provided that $a<\frac12$. The same claim holds also in the  region $|\xi|\gg|\xi_1|$ by  letting $ \rho=\xi^5-(\xi-\xi_1)^5$, and noting that
$$d\rho\approx |\xi_1||\xi|^3 d\xi\approx |\rho|^{\frac34}|\xi_1|^{\frac14} d\xi.$$

In the region $|\theta|\gg  \la \xi\ra$,  by symmetry we can assume that $|\theta-\theta_1|\gtrsim |\theta|$, and hence we have the bound  
$$ \int  \frac{|\xi|^{\frac52} \big|\frac{\theta_1\xi}{\xi_1}\pm\sqrt{\tfrac{|\eta\xi (\xi-\xi_1)|}{|\xi_1|}}\big|^{ 2a} |\xi_1|^{-\frac12}|\xi-\xi_1|^{\frac12} d\xi  d\eta   }{\la  \theta_1\ra^{2s}  \la M\ra^{1-} |\eta|^{\frac12}  \la\lambda_1 -\xi^5+(\xi-\xi_1)^5+\frac{\theta_1^2}{\xi_1 }-\eta\ra^{1-} }
$$
When $|\xi|\les |\xi_1|$, for $a\leq \min(s,\frac12)$, we bound this by
$$
\int  \frac{|\xi|^{3}    d\xi    }{   |\xi|^{3-} |\xi_1|^{1-}|\xi-\xi_1|^{1-}   }\les 1.
$$
When $|\xi| \gg |\xi_1|$, for $a<\min(s,\frac12)$,  we have the bound 
\begin{multline*}
\int \frac{|\xi|^{2a-1+}     d\xi  d\eta   }{   |\xi_1|^{\frac32+a-} |\eta|^{\frac12-a}  \la\lambda_1 -\xi^5+(\xi-\xi_1)^5+\frac{\theta_1^2}{\xi_1 }-\eta\ra^{1-} } 
\\
\les \int \frac{|\xi|^{2a-1+}     d\xi     }{   |\xi_1|^{\frac32+a-}  \la\lambda_1 -\xi^5+(\xi-\xi_1)^5+\frac{\theta_1^2}{\xi_1 } \ra^{\frac12-a-} }.
\end{multline*}
Once again, by letting $ \rho=\xi^5-(\xi-\xi_1)^5$ as above, we bound this by
$$
\int \frac{   d\rho    }{   |\xi_1|^{\frac32 +\frac{3a}2-} |\rho|^{1-\frac{a}2-} \la\lambda_1 -\rho+\frac{\theta_1^2}{\xi_1 } \ra^{\frac12-a-} }\les 1 
$$
provided that $a<\frac13$.

We now consider the case $|\xi_1|<1$, $|\xi|\gg1$. 
By the Cauchy--Schwarz inequality, the convolution structure, and by performing the $\lambda_1, \lambda$ integrals, it suffices to prove
$$
\sup_{\xi,\theta}\int \int_{|\xi_1|<1} \frac{\xi^2\la\xi,\theta\ra^{2s+2a}}{\la\theta_1\ra^{2s}\la\xi,\theta-\theta_1\ra^{2s}\la M\ra^{1-}}d\xi_1d\theta_1\les 1,
$$ 
where 
$$M\approx \xi^4|\xi_1| +\eta^2\ \text{and}\ \eta^2=\frac{(\theta\xi_1-\theta_1\xi)^2}{ \xi^2|\xi_1| }.$$
Note that 
$$
|\eta| d\eta\approx \frac{|\theta\xi_1-\theta_1\xi| }{ |\xi \xi_1| } d\theta_1\approx |\eta||\xi_1|^{-\frac12} d\theta_1.
$$ 
Therefore, we write the integral as
$$
\int \int_{|\xi_1|<1} \frac{\xi^2\la\xi,\theta\ra^{2s+2a}|\xi_1|^{\frac12}}{\la\theta_1\ra^{2s}\la\xi,\theta-\theta_1\ra^{2s} \la |\eta|+\xi^2|\xi_1|^{\frac12}\ra^{2-}}d\xi_1d\eta,
$$
where $\theta_1=\frac{\theta\xi_1}\xi\pm\eta|\xi_1|^{\frac12}$.

In the case $|\theta|\les |\xi|$, we estimate this by
$$
\int \int_{|\xi_1|<1} \frac{\xi^{2+2a} |\xi_1|^{\frac12}}{\la \eta|\xi_1|^{\frac12}\ra^{2s}  \la |\eta|+\xi^2|\xi_1|^{\frac12}\ra^{2-}}d\xi_1d\eta,
$$
Since the integral decreases in $s$  we can assume that  $s< \frac12$.  We bound the integral by
$$
 \int_{|\xi_1|<1} \frac{\xi^{2+2a} |\xi_1|^{\frac12-s}}{   ( \xi^2|\xi_1|^{\frac12})^{1+2s-}}d\xi_1 \les |\xi|^{2a-4s+}\les 1
$$
provided that $a<2s.$ For $s\geq\frac12$, we use the bound above for $s=\frac12-$ with $a<1$.

In the case $|\theta|\gg|\xi|$ and $ | \theta|\gg |\theta_1|$,  we have 
$$
\int \int_{|\xi_1|<1} \frac{\xi^{2} |\theta|^{ 2a} |\xi_1|^{\frac12}}{\la \frac{\theta\xi_1}\xi \pm\eta|\xi_1|^{\frac12}\ra^{2s}  \la |\eta|+\xi^2|\xi_1|^{\frac12}\ra^{2-}}d\xi_1d\eta,
$$
For $s<\frac12$ we use the bound
$$
\int \int_{|\xi_1|<1} \frac{\xi^{2} |\theta|^{ 2a} |\xi_1|^{\frac12-s}}{|\frac{\theta|\xi_1|^{\frac12}}\xi\pm\eta |^{2s} \la  \eta \ra^{1-2s+2a} ( \xi^2|\xi_1|^{\frac12})^{1+2s-2a-}}d\xi_1d\eta
\les \int_{|\xi_1|<1} \frac{\xi^{2+2a}   |\xi_1|^{\frac12-s-a}}{  ( \xi^2|\xi_1|^{\frac12})^{1+2s-2a-}}d\xi_1\les 1,
$$
provided that $a<\frac{2s}3$. For $s\geq\frac12$ we use the bound above for $s=\frac12-$ with $a<\frac13$.

It remains to consider the case $|\theta|\gg|\xi|$ and $ | \theta|\les |\theta_1|$. When $s<\frac12$, we have 
$$
\int \int_{|\xi_1|<1} \frac{\xi^2 |\theta|^{2a}|\xi_1|^{\frac12}}{|\theta(1-\frac{\xi_1}\xi)\pm\eta|\xi_1|^{\frac12}|^{2s} \la |\eta|+\xi^2|\xi_1|^{\frac12}\ra^{2-}}d\xi_1d\eta \les \int_{|\xi_1|<1} \frac{\xi^2  |\xi_1|^{\frac12-s+a}}{ ( \xi^2|\xi_1|^{\frac12})^{1+2s-2a-}}d\xi_1\les 1,
$$
provided that $a<s$. Once again, for $s\geq \frac12$ we use the bound for $s=\frac12-$ with $a<\frac12$.

The remaining case $|\xi|,|\xi_1|<1$ is treated similarly but easier.
\end{proof}

We now consider nonlinear estimates involving the norm 
$$
\|F\|_{X^{\frac12+, \frac{s}2-\frac14,-b_1}}=\big\|\la\xi,\theta\ra^{\frac12+} |\xi|^{ \frac{s}2-\frac14} 
\la\lambda-\xi^5+\tfrac{\theta^2}\xi\ra^{-b_1} \widehat F(\xi,\theta,\lambda)\big\|_{L^2_{\xi,\theta,\lambda}},
$$
where $b_1=\tfrac34-\tfrac{s}2$. Recall that this norm appears only when the Sobolev index, $s+a$, is at least $\frac12$. Therefore, for $a$ we have the lower bound $\frac12-s$. Together with the upper bound for $a$ we see that the relevant range for $s$ is  $\tfrac3{10}<s<\frac52$.
\begin{theorem}\label{thm:nonlin2}
Fix $\tfrac3{10}<s<\frac52$ and $\frac12-s<a<\min(\frac{2s}3,\frac13,\frac32-\frac{3s}5)$ and let $b_1=\tfrac34-\tfrac{s+a}2$. If $b<\tfrac12$ is sufficiently close to $\tfrac12$, then
$$
\|(u^2)_x\|_{X^{\frac12+, \frac{s+a}2-\frac14,-b_1}}\les \|u\|_{X^{s,b}}^2.
$$
\end{theorem}
\begin{proof}
Using the notation of the previous proof  it suffices to prove that
\be\label{eq:ntp}
\sup_{\xi,\theta,\lambda} \int  \frac{|\xi|^{s+a+\frac32} \la \xi,\theta\ra^{1+} \la\tau\ra^{s+a -\tfrac32  } d\xi_1d\theta_1d\lambda_1  }{\la\xi_1,\theta_1\ra^{2s }\la\xi-\xi_1,\theta-\theta_1\ra^{2s} \la \tau_1\ra^{1-}  \la  \tau_2\ra^{1-} } <\infty.
\ee
Below, we only consider the case $|\xi|\gg 1$;   the case $|\xi|\les 1$ is easier and will be omitted.

We first consider the case $\tfrac32\leq s+a<\tfrac52$. Since $a<\frac13$, we can assume that $s>1$. We investigate this in two parts: 

1. $\la \tau_1\ra$ and $\la \tau_2\ra \ll M$, 

2. $\la \tau_1\ra$ or $\la \tau_2\ra \gtrsim \max(M,\la \tau\ra)$, 

where 
$$M:= |\xi\xi_1(\xi-\xi_1)|  (\xi^2+\xi_1^2) +| \eta|,$$
$$\eta=\tfrac{\theta_1^2}{\xi_1}+\tfrac{(\theta-\theta_1)^2}{\xi-\xi_1}-\tfrac{\theta^2}{\xi}=  \tfrac{[(\theta-\theta_1)\xi_1-\theta_1(\xi-\xi_1)]^2}{\xi\xi_1(\xi-\xi_1)}=(\tfrac{\theta-\theta_1}{\xi-\xi_1}-\tfrac{\theta_1}{\xi_1})^2 \tfrac{\xi_1(\xi-\xi_1)}{\xi}.$$

\noindent
\textsc{Case 1:} $\la \tau_1\ra$ and $\la \tau_2\ra \ll M$. 

In this case we have $| \tau|\approx M$. Integrating in $\lambda_1$, we estimate the integral in \eqref{eq:ntp} by
$$
\int  \frac{|\xi|^{s+a+\frac32} \la \xi,\theta\ra^{1+} \la M\ra^{s+a -\tfrac32  } d\xi_1d\theta_1   }{\la\xi_1,\theta_1\ra^{2s }\la\xi-\xi_1,\theta-\theta_1\ra^{2s} \big\la \lambda +\frac{\theta^2}\xi-\xi_1^5-(\xi-\xi_1)^5+\eta \big\ra^{1-} }.
$$
We consider several subcases when $|\xi|\gg 1$.

Subcase 1.i: $ |\xi_1|\ll  |\xi|$.\footnote{The case $ |\xi-\xi_1|\ll  |\xi|$ is similar by symmetry.} 

In this case we have
$$
|\xi-\xi_1|\approx|\xi|,\quad M\approx \xi^4|\xi_1|+|\eta|,\quad |\eta|\approx (\tfrac{\theta-\theta_1}{\xi-\xi_1}-\tfrac{\theta_1}{\xi_1})^2 |\xi_1|.
$$
Therefore, it is reasonable to consider the following regions:\\
Region 1: $|\eta|\ll \la \xi^4\xi_1\ra$ and $|\theta_1|\ll|\xi_1|\xi^2$,\\
Region 2: $|\eta|\ll\la \xi^4\xi_1\ra$, $|\theta_1|\gtrsim|\xi_1|\xi^2$, and $|\theta-\theta_1|\gtrsim|\xi|^3$,\\
Region 3: $|\eta|\gtrsim \la \xi^4\xi_1\ra$ and $|\theta_1|\gtrsim |\eta\xi_1|^{\frac12}\gtrsim |\xi_1|\xi^2$,\\
Region 4: $|\eta|\gtrsim \la \xi^4\xi_1\ra$, $|\theta_1|\ll |\eta\xi_1|^{\frac12}  $ and $|\theta-\theta_1 |\gtrsim \tfrac{|\xi||\eta|^{\frac12}}{|\xi_1|^\frac12}\gtrsim |\xi|^3$.

In Region 1, letting  $\rho=-\xi_1^5-(\xi-\xi_1)^5+\eta$ in the $\xi_1$ integral, we note that
\be\label{rhochange}
|\rho|\approx |\xi|^5\gg |\eta|,\quad d\rho=\Big| -5\xi_1^4+5(\xi-\xi_1)^4 -\tfrac{\theta_1^2}{\xi_1^2} +\tfrac{(\theta-\theta_1)^2}{(\xi-\xi_1)^2}\Big| d\xi_1\gtrsim  \xi^4  d\xi_1.
\ee
Therefore we can estimate the integral by
\begin{multline*}
\int  \frac{ |\xi|^{s+a-\frac52+} |\xi|^{4s+4a -6  } \la \xi,\theta\ra^{1+}  d\rho d\theta_1   }{\la \theta_1\ra^{s-a+\frac32  }\la\xi ,\theta-\theta_1\ra^{2s} \big\la \lambda +\frac{\theta^2}\xi-\rho  \big\ra^{1-} \la \rho\ra^{0+} } 
\\
\les \int  \frac{|\xi|^{3s+5a-\frac{15}2+}\la \xi,\theta\ra^{1+}    d\theta_1   }{\la \theta_1\ra^{s-a+\frac32 }\la\xi ,\theta-\theta_1\ra^{1+}   }\les  |\xi|^{3s+5a-\frac{15}2+}\les 1,
\end{multline*}
provided that $s-a+\frac32 >1$ and $3s+5a-\frac{15}2<0$. We thus need $a<\min(s+\frac12,\frac32-\frac{3s}5)$.

In Region 2, we estimate the integral by (using $2s>2$,   $|\theta_1|\gtrsim |\xi_1|   \xi^2$, and $|\theta-\theta_1|\gtrsim |\xi|^3$)
$$
\int  \frac{|\xi|^{s+a+\frac32}  |\xi|^{4s+4a -6  }  |\xi_1|^{s+a -\tfrac32  }  \la  \xi,\theta\ra^{1+}d\xi_1d\theta_1   }{ |\xi_1|   \xi^2 \la \theta_1\ra^{1+ } |\xi|^{6s-3-} \la\xi, \theta-\theta_1\ra^{1+}  }
\les |\xi|^{5a- s-\frac72+} \int_{  |\xi_1|\ll|\xi|} |\xi_1|^{ s+a-\frac52} d\xi_1\les 
 |\xi|^{6a-5  +} ,
$$
which is bounded provided that  $a<\frac56$. 
  
In Region 3, passing to $\eta$ variable in $\theta_1$ integral (using \eqref{etatheta1}), we have the bound
\begin{multline*}
\int  \frac{|\xi|^{s+a+\frac32}  |\xi_1|^{\frac12} |\eta|^{s+a -\tfrac32  } d\xi_1d\eta }{\la\eta\xi_1\ra^{s-\frac12-} |\xi |^{2s-1-}  |\eta|^{ \tfrac12}\big\la \lambda +\frac{\theta^2}\xi-\xi_1^5-(\xi-\xi_1)^5+\eta \big\ra^{1-} }
\\ \les \int  \frac{|\xi|^{-s+a+\frac52+}  |\xi_1|^{\frac12}  d\xi_1d\eta }{|\eta\xi_1|^{\frac{s+a}2-\frac14}   |\eta|^{2-s-a}\big\la \lambda +\frac{\theta^2}\xi-\xi_1^5-(\xi-\xi_1)^5+\eta \big\ra^{1-} }.
\end{multline*}
Integrating in $\eta$ using $|\eta|\gtrsim \la  \xi^4 \xi_1\ra $, we have
$$
\les \int_{|\xi_1|\ll|\xi| } \frac{|\xi|^{-s+a+\frac52+}   |\xi_1|^{\frac12} d\xi_1  }{( \xi^4\xi_1^2)^{\frac{s+a}2-\frac14}   |\xi^4\xi_1|^{2-s-a-} }=|\xi|^{3a+s-\frac92+}\int_{|\xi_1|\ll|\xi| }    |\xi_1|^{-1+} d\xi_1 \les 1
$$
provided that $a<\frac32-\frac{s}3$.

In Region 4,  noting that $|\theta|\gg| \theta_1|$ and using $|\theta|\approx|\theta-\theta_1|\gtrsim \tfrac{|\xi||\eta|^{\frac12}}{|\xi_1|^\frac12}\gtrsim|\xi|^3$ and $|\eta|\gtrsim\xi^4|\xi_1|$,  we have
\begin{multline*}
\int  \frac{|\xi|^{s+a+\frac32}     |\eta|^{s+a -\frac32 } d\xi_1d\theta_1  }{\la\xi_1,\theta_1\ra^{2s }\la\theta \ra^{2s-1-}  }
\les \int  \frac{|\xi|^{-s+a+\frac52+}    |\xi_1|^{s-\frac12+}    d\xi_1  d\theta_1}{\la\xi_1,\theta_1 \ra^{2s }|\eta|^{1-a}  }
\\ \les \int  \frac{|\xi|^{-s+a+\frac52+}   |\xi_1|^{s-\frac12+}      d\xi_1  }{\la\xi_1\ra^{2s -1}|\xi^4\xi_1|^{1-a}  }\les |\xi|^{5a-s-\frac32+}\les 1,
\end{multline*}
 provided that  $a<\min(\frac{s}5+\frac3{10}, s-\frac12)$.

Subcase 1.ii: $ |\xi_1|\approx |\xi-\xi_1|\gg  |\xi|$.  

In this case we have
$$
M\approx \xi_1^4|\xi |+|\eta|,\quad |\eta|\approx  (\tfrac{\theta-\theta_1}{\xi-\xi_1}-\tfrac{\theta_1}{\xi_1})^2\tfrac{\xi_1^2}{|\xi|}.
$$
This leads to the following regions\footnote{Region 1 and 2 suffice to cover the case $|\eta|\ll\xi_1^4|\xi|$ by chosing the implicit constants carefully.}:\\
Region 1: $|\eta|\ll\xi_1^4|\xi|$, and $|\theta_1|, |\theta-\theta_1|\ll|\xi|\xi_1^2$,\\
Region 2: $|\eta|\ll\xi_1^4|\xi|$, and $|\theta_1|, |\theta-\theta_1|\gtrsim|\xi|\xi_1^2$, \\
Region 3: $|\eta|\gtrsim \xi_1^4|\xi|$, and $|\theta_1|\gtrsim |\eta\xi |^{\frac12}\gtrsim |\xi |\xi_1^2$,\\
Region 4: $|\eta|\gtrsim \xi_1^4|\xi|$, $|\theta_1|\ll |\eta\xi |^{\frac12}  $, and $|\theta-\theta_1 |\gtrsim |\eta\xi |^{\frac12}\gtrsim |\xi |\xi_1^2$.

In Region 1, we bound the integral by
$$
\int \frac{|\xi|^{2s+2a } \la \xi,\theta\ra^{1+} |\xi_1|^{4s+4a -6  } d\xi_1d\theta_1   }{ |\xi_1|^{4s-2-}\la\xi ,\theta_1\ra^{1+ }\la\xi ,\theta-\theta_1\ra^{1+} \big\la \lambda +\frac{\theta^2}\xi-\xi_1^5-(\xi-\xi_1)^5+\eta \big\ra^{1-} }.
$$
Letting  $\rho=-\xi_1^5-(\xi-\xi_1)^5+\eta$ in the $\xi_1$ integral, we note that
$$
|\rho|\approx \xi_1^4 |\xi| \gg |\eta|,
$$
$$
d\rho=\Big| -5\xi_1^4+5(\xi-\xi_1)^4 -\tfrac{\theta_1^2}{\xi_1^2} +\tfrac{(\theta-\theta_1)^2}{(\xi-\xi_1)^2}\Big| d\xi_1\gtrsim   |\xi_1|^3 |\xi| d\xi_1 \gtrsim\xi^4  d\xi_1.
$$ 
Therefore we can estimate the integral by (for $a<1$)
$$
\int  \frac{|\xi|^{2s+2a }|\xi|^{ 4a -8+  } \la \xi,\theta\ra^{1+}  d\rho d\theta_1   }{  \la\xi ,\theta_1\ra^{1+ }\la\xi ,\theta-\theta_1\ra^{1+} \big\la \lambda +\frac{\theta^2}\xi+\rho\big\ra^{1-} | \rho|^{0+} }\les |\xi|^{2s+6a   -8+  } \les 1,
$$
provided that   $a<\min(1,\frac43-\frac{ s}3)$.

In Region 2, we estimate the integral by 
\begin{multline*}
\int  \frac{|\xi|^{2s+2a}   |\xi_1|^{4s+4a -6 }  \la \xi,\theta\ra^{1+} d\xi_1d\theta_1   }{(|\xi|\xi_1^2)^{4s-2-}   \la\xi ,\theta_1\ra^{1+ }\la  \xi ,\theta-\theta_1\ra^{1+} } \\ \les  
\int_{|\xi_1|\gg|\xi|}  |\xi|^{2-2s+2a+}   |\xi_1|^{-4s+4a -2 +}    d\xi_1 \les |\xi|^{1-6s+6a +}\les 1 ,
\end{multline*}
provided that $-4s+4a-2 <-1$ and $1-6s+6a <0$.  We thus need $a< s-\frac16 $.

In Region 3, passing to $\eta$ variable in $\theta_1$ integral, we have the bound
\begin{multline*}
\int  \frac{|\xi|^{s+a+\frac32} \la \xi,\theta\ra^{1+}|\xi_1|^{\frac12} |\eta|^{s+a -\tfrac32  } d\xi_1d\eta }{|\eta\xi |^{s}\big\la\xi_1 ,\theta-\frac{\theta\xi_1}{\xi}\pm\sqrt{\tfrac{|\eta\xi_1(\xi-\xi_1)|}{|\xi|}}\big\ra^{2s}  |\eta|^{ \tfrac12}\big\la \lambda +\frac{\theta^2}\xi-\xi_1^5-(\xi-\xi_1)^5+\eta \big\ra^{1-} }
\\ = \int  \frac{|\xi|^{ a+\frac32}|\xi_1|^{\frac12}   \la \xi,\theta\ra^{1+}   d\xi_1d\eta }{ \big\la\xi_1 ,\theta-\frac{\theta\xi_1}{\xi}\pm\sqrt{\tfrac{|\eta\xi_1(\xi-\xi_1)|}{|\xi|}}\big\ra^{2s}  |\eta|^{2-a}\big\la \lambda +\frac{\theta^2}\xi-\xi_1^5-(\xi-\xi_1)^5+\eta \big\ra^{1-} }.
\end{multline*}
Noting that
\begin{multline*}
\big\la\xi_1 ,\theta-\tfrac{\theta\xi_1}{\xi}\pm\sqrt{\tfrac{|\eta\xi_1(\xi-\xi_1)|}{|\xi|}}\big\ra^{2s}  |\eta|^{2-a}\gtrsim |\xi_1|^{2s-1-}\la\xi,\theta\ra^{1+}|\eta|^{\frac32-a-}|\xi|^{\frac12+}|\xi_1|^{-1 -}\\ \gtrsim |\xi_1|^{2s+4-4a-}|\xi|^{2-a-}\la\xi,\theta\ra^{1+}  |\eta|^{0+},
\end{multline*}
we estimate the integral by
$$
\int_{|\xi_1|\gg|\xi|} \frac{|\xi|^{2a-\frac12+}    d\xi_1  }{|\xi_1|^{2s-4a+\frac72-}    }\les 1,
$$
provided that $a<\min(\frac{s}2+\frac58 ,\frac{s}3+\frac12)$.

In Region 4,  noting that $|\theta|\gg| \theta_1|$ and using $|\theta|\gtrsim  |\xi \eta|^{\frac12} $ and $|\eta|\gtrsim\xi_1^4|\xi |$,  we have the bound
$$ 
\int  \frac{|\xi|^{s+a+\frac32}     |\eta|^{s+a -\frac32 } d\xi_1d\theta_1  }{\la\xi_1,\theta_1\ra^{2s }\la\theta \ra^{2s-1-}  }
\les \int  \frac{|\xi|^{ a+ 2+}      d\xi_1  d\theta_1}{\la\xi_1,\theta_1 \ra^{2s }|\eta|^{1-a-}  }
  \les \int_{|\xi_1|\gg|\xi| } \frac{|\xi|^{2a+1+}        d\xi_1  }{|\xi_1|^{2s -4a+3 }  }\les 1,
$$
 provided that  $a< \min(1,\frac{s}3+\frac16)$.

Subcase 1.iii:  $|\xi_1|\approx |\xi-\xi_1| \approx |\xi|$.

In this case we have
$$
M\approx |\xi|^5 +|\eta|,\quad |\eta|\approx  (\tfrac{\theta-\theta_1}{\xi-\xi_1}-\tfrac{\theta_1}{\xi_1})^2 |\xi| .
$$
In this subcase, it suffices to consider 3 regions:\\
Region 1: $|\eta|\ll |\xi|^5$,\\
Region 2:   $|\eta|\gtrsim |\xi|^5$ and $|\theta_1|\gtrsim |\eta\xi |^{\frac12}\gtrsim |\xi|^3$,\\
Region 3: $|\eta|\gtrsim |\xi|^5$, $|\theta_1|\ll |\eta\xi |^{\frac12}  $ and $|\theta-\theta_1 |\gtrsim |\eta\xi |^{\frac12}\gtrsim |\xi|^3$.

In Region 1, using 
$$\la \xi ,\theta_1\ra^{2s}\la \xi,\theta-\theta_1\ra^{2s}  \gtrsim \la \xi,\theta \ra^{1+} |\xi|^{4s-1-},$$ and passing to $\eta $ variable in $\theta_1$ integral, we have the bound
\be \label{rhotrick}
\int  \frac{|\xi|^{2s+6a -\frac{9}2+ }     d\xi_1d\eta }{  |\eta|^{ \tfrac12}\big\la \lambda +\frac{\theta^2}\xi-\xi_1^5-(\xi-\xi_1)^5+\eta \big\ra^{1-} } \les 
\int_{|\xi_1|\approx |\xi|}  \frac{|\xi|^{2s+6a -\frac{9}2 + }     d\xi_1  }{  \big\la \lambda +\frac{\theta^2}\xi-\xi_1^5-(\xi-\xi_1)^5  \big\ra^{\tfrac12-} }.
\ee
Using H{\oo}lder's Inequality, this is 
$$
\les |\xi|^{2s+6a -\frac{9}2 } |\xi|^{0+} \Big[\int \frac{  |\xi-2\xi_1|   d\xi_1  }{  \big\la \lambda +\frac{\theta^2}\xi-\xi_1^5-(\xi-\xi_1)^5  \big\ra^{1+} }\Big]^{\frac12-}.
$$
Letting $\rho=\xi_1^5+(\xi-\xi_1)^5$ in the $\xi_1$ integral and noting that 
$$
|\rho|\approx |\xi|^5,\quad d\rho =|5\xi_1^4-5(\xi-\xi_1)^4|d\xi_1\gtrsim |\xi|^3|\xi-2\xi_1| d\xi_1 ,
$$ 
we estimate the integral by $|\xi|^{2s+6a -6+}\les 1 $, provided that $a<1 -\frac{s}3$.

Regions 2 and 3 are identical to Regions 3 and 4 of Subcase 1.ii, respectively. 

\noindent
\textsc{Case 2:} $\la \tau_1\ra$ or $\la \tau_2\ra \gtrsim \max(M,\la \tau\ra)$.

By symmetry, we assume that $\la \tau_1\ra \gtrsim \max(M,\la \tau\ra )$. Integrating in $\lambda_1$, we estimate the integral in \eqref{eq:ntp} by
$$
\int  \frac{|\xi|^{s+a+\frac32} \la \xi,\theta\ra^{1+}  d\xi_1d\theta_1   }{\la\xi_1,\theta_1\ra^{2s }\la\xi-\xi_1,\theta-\theta_1\ra^{2s} \la M\ra^{\frac52-s-a -} }.
$$

If $|\xi_1|, |\xi-\xi_1|\gtrsim 1$, using $|M|\gtrsim |\xi|^3|\xi_1||\xi-\xi_1| \gtrsim |\xi|^4$, the integral above has the upper bound
$$
\int  \frac{|\xi|^{5s+5a+\frac32-10+} \la \xi,\theta\ra^{1+}  d\xi_1d\theta_1   }{\la\xi_1,\theta_1\ra^{2s }\la\xi-\xi_1,\theta-\theta_1\ra^{2s}  }.
$$
Noting that $\max(\la\xi_1,\theta_1\ra,\la\xi-\xi_1,\theta-\theta_1\ra)\gtrsim \la \xi,\theta\ra$, we further bound this by
$$
\int  \frac{|\xi|^{5s+5a+\frac32-10+} |\xi|^{-2s+1+}   d\xi_1d\theta_1   }{\min(\la\xi_1,\theta_1\ra,\la\xi-\xi_1,\theta-\theta_1\ra)^{2s}  } \les |\xi|^{3s+5a-\frac{15}2+},
$$
which is $\lesssim 1$ for $|\xi|>1$ provided that $a<\frac32-\frac{3s}5$. In the last inequality we used $s>1$ to integrate in $\xi_1 $ and $\theta_1$.

When $|\xi_1|<1$, we have $|M|\gtrsim \xi^4|\xi_1|$, which leads to
$$
\int  \int_{|\xi_1|<1} \frac{|\xi|^{5s+5a+\frac32-10+} \la \xi,\theta\ra^{1+}  d\xi_1d\theta_1   }{\la \theta_1\ra^{2s }\la\xi ,\theta-\theta_1\ra^{2s} |\xi_1|^{\frac52-s-a -} }\les |\xi|^{5s+5a+\frac32-10+}\la \xi,\theta\ra^{1-2s+}\les |\xi|^{3s+5a-\frac{15}2+},
$$ 
which is bounded for $|\xi|>1$ provided that $a<\frac32-\frac{3s}5$. The case $|\xi-\xi_1|\les 1$ is similar by symmetry.

Now, we consider the case $\frac12<s+a<\frac32$. 
We first consider the resonance case $|\xi_1|<1$. The case $|\xi-\xi_1|<1$ is similar. It suffices to bound 
\begin{multline*}
 \int \int_{|\xi_1|<1} \frac{|\xi|^{s+a+\frac32} \la \xi,\theta\ra^{1+}  d\xi_1d\theta_1   }{\la \theta_1\ra^{2s }\la\xi ,\theta-\theta_1\ra^{2s} \la M\ra^{\frac32-s-a } \big\la \lambda +\frac{\theta^2}\xi-\xi_1^5-(\xi-\xi_1)^5+\eta \big\ra^{1-}  }\\+ \int \int_{|\xi_1|<1} \frac{|\xi|^{s+a+\frac32} \la \xi,\theta\ra^{1+}  d\xi_1d\theta_1   }{\la \theta_1\ra^{2s }\la\xi ,\theta-\theta_1\ra^{2s} \la M\ra^{1- }   } .
\end{multline*}
If $s>\frac12$ or $|\xi|\gtrsim |\theta|$, noting that 
$M>\xi^4|\xi_1|$,
we bound the second integral by
 $$
\les  \int_{|\xi_1|<1} \frac{|\xi|^{s+a+\frac32} \la \xi \ra^{1-\min(2s,4s-1))+}  d\xi_1   }{  |\xi|^{4-}|\xi_1|^{1-  }   } \les |\xi|^{ s+ a-\frac32-\min(2s,4s-1))+}\les 1
$$
provided that $a<\frac32-s+ \min(2s,4s-1)$.  We bound the first integral by
$$
\les  \int \int_{|\xi_1|<1} \frac{|\xi|^{s+a+\frac32} \la \xi \ra^{1- 2s +}  d\xi_1  d\theta_1 }{  \la M\ra^{\frac32-s-a } \big\la \lambda +\frac{\theta^2}\xi-\xi_1^5-(\xi-\xi_1)^5+\eta \big\ra^{1-}   }.
$$
Passing to $\eta $ variable
$$
\les  \int \int_{|\xi_1|<1} \frac{|\xi|^{s+a+\frac32} \la \xi \ra^{1- 2s +} |\xi_1|^{\frac12} d\xi_1  d\eta  }{  \la \xi^4\xi_1\ra^{\frac32-s-a } |\eta|^{\frac12}\big\la \lambda +\frac{\theta^2}\xi-\xi_1^5-(\xi-\xi_1)^5+\eta \big\ra^{1-}   } 
$$
$$
\les \int_{|\xi_1|<1} \frac{|\xi|^{s+a+\frac32} \la \xi \ra^{1- 2s +} |\xi_1|^{\frac12} d\xi_1   }{  \la \xi^4\xi_1\ra^{\frac32-s-a }  \big\la \lambda +\frac{\theta^2}\xi-\xi_1^5-(\xi-\xi_1)^5  \big\ra^{\frac12-}   }. 
$$
Letting $\rho= \xi_1^5+(\xi-\xi_1)^5$ we have
$$
\les \int_{|\rho-\xi^5|\les \xi^4} \frac{|\xi|^{s+a+\frac32} \la \xi \ra^{1- 2s +}  d\rho   }{  |\xi |^{4\min(\frac12,\frac32-s-a) }  \big\la \lambda +\frac{\theta^2}\xi-\rho \big\ra^{\frac12-}   \xi^4}\les |\xi|^{ a-s+\frac12-4\min(\frac12,\frac32-s-a)+} 
$$
by H{\oo}lder's Inequality. This is bounded if $a\leq 1-s$ or if $1-s<a<  \frac{11}{10}-\frac{3s}5 $.

If $s<\frac12$ and $|\theta|\gg|\xi|$,  
without loss of generality $\la\xi ,\theta-\theta_1\ra \gtrsim \la\xi ,\theta \ra$, and hence we can bound both integrals by
$$
\int \int_{|\xi_1|<1} \frac{|\xi|^{s+a+\frac32} \la  \theta\ra^{1-2s+}  d\xi_1d\theta_1   }{\la \theta_1\ra^{2s } \la M\ra^{\frac32-s-a }   }.
$$
 Passing to $\sigma=\sqrt{|\eta|}$, we get
 $$
 \int \int_{|\xi_1|<1}\frac{|\xi|^{s+a+\frac32} |\xi_1|^{\frac12} \la \theta\ra^{1-2s+} d\xi_1d\sigma}{\la \theta_1\ra^{2s}    \la \xi^4|\xi_1|+\sigma^2 \ra^{\frac32-a-s-} }
 $$
 where $\theta_1=\frac{\theta\xi_1}{\xi}\pm \sigma \sqrt{\tfrac{| \xi_1(\xi-\xi_1)|}{|\xi|}}$. We rewrite this as
  $$
 \int \int_{|\xi_1|<1}\frac{|\xi|^{s+a+\frac32} |\xi_1|^{\frac12-s} \la \theta\ra^{1-2s+} d\xi_1d\sigma}{\big| \theta \frac{|\xi_1|^{\frac12}}{|\xi |^{\frac12}|\xi-\xi_1|^{\frac12}}\pm \sigma \big|^{2s}   \la \sigma  \ra^{2-4s+} \la \xi^4|\xi_1| \ra^{\frac12-a+s-} }  $$   
 $$\les \int_{|\xi_1|<1} \frac{|\xi|^{s+a+\frac32} |\xi|^{1-2s+} |\xi_1|^{0-}   d\xi_1 }{  |\xi|^{2-4a+4s-} |\xi_1|^{\frac12-a+s-} } \les |\xi|^{5a-5s+\frac12+}
 $$
 which is bounded provided that $a<s-\frac1{10}$.

In the nonresonance cases, i.e. $|\xi_1|,|\xi-\xi_1|\gg 1$, we consider two  cases: 

1. $\la \tau_1\ra$ and $\la \tau_2\ra \ll M$,

2. $\la \tau_1\ra$ or $\la \tau_2\ra \gtrsim \max(M,\la \tau\ra)$.

\noindent
\textsc{Case 1:} $\la \tau_1\ra$ and $\la \tau_2\ra \ll M$. 

In this case we have $| \tau|\approx M$. Integrating in $\lambda_1$, we estimate the integral in \eqref{eq:ntp} by
\be\label{inttemp5} 
\int  \frac{|\xi|^{s+a+\frac32} \la \xi,\theta\ra^{1+}  d\xi_1d\theta_1   }{\la\xi_1,\theta_1\ra^{2s }\la\xi-\xi_1,\theta-\theta_1\ra^{2s} \la M\ra^{\frac32-s-a } \big\la \lambda +\frac{\theta^2}\xi-\xi_1^5-(\xi-\xi_1)^5+\eta \big\ra^{1-} }.
\ee
We estimate this integral by considering two subcases.

Subcase 1.i: $ 1\ll |\xi_1|\ll  |\xi|$.\footnote{The case $ 1\ll |\xi-\xi_1|\ll  |\xi|$ is similar by symmetry.} 

In this case we have
$$
|\xi-\xi_1|\approx|\xi|,\quad M\approx \xi^4|\xi_1|+|\eta|,\quad |\eta|\approx (\tfrac{\theta-\theta_1}{\xi-\xi_1}-\tfrac{\theta_1}{\xi_1})^2 |\xi_1|.
$$
Therefore, it is reasonable to consider the following regions:\\
Region 1: $|\eta|\ll  \xi^4|\xi_1|$ and $|\theta_1|\ll|\xi_1|\xi^2$,\\
Region 2: $|\eta|\ll\xi^4|\xi_1|$, $|\theta_1|\gtrsim|\xi_1|\xi^2$, and $|\theta-\theta_1|\gtrsim|\xi|^3$,\\
Region 3: $|\eta|\gtrsim \xi^4|\xi_1|$. 

 In  Region 1, we first note that $|\theta-\theta_1|\ll|\xi|^3$, and hence $|\theta|\ll|\xi|^3$. Using the change of variable in \eqref{rhochange}, we bound \eqref{inttemp5} by
$$
\int_{|\theta_1|\les |\xi|^3}  \frac{|\xi|^{s+a+\frac32} \la \xi,\theta\ra^{1+} |\xi|^{4s+4a -6 }    d\rho  d\theta_1}{\la \theta_1\ra^{2s }\la\xi  ,\theta-\theta_1\ra^{2s} \big\la \lambda +\frac{\theta^2}\xi-\rho \big\ra^{1-}  |\xi|^{4-}\rho^{0+}} 
 \les |\xi|^{5s+5a-\frac{17}2+} \la \xi,\theta\ra^{1-\min(2s,4s-1)+}.  
$$
 This is acceptable by considering the cases $s>\frac12$ and $\frac14<s <\frac12$ separately and using $|\theta|\ll|\xi|^3$.

Region 2: In this region we have $|\theta_1|\approx|\frac{\xi_1\theta}{\xi}|$ which also implies $ |\theta|\approx |\theta-\theta_1|\gtrsim |\xi|^3$. Therefore, we 
  bound \eqref{inttemp5} by
$$
\int  \frac{|\xi|^{s+a+\frac32} \la  \theta\ra^{1-4s+} |\xi|^{4s+4a -6 } |\xi_1|^{s+a -\tfrac32  } |\xi|^{2s} d\xi_1d\theta_1   }{|\xi_1|^{2s }  \big\la \lambda +\frac{\theta^2}\xi-\xi_1^5-(\xi-\xi_1)^5+\eta \big\ra^{1-} }.
$$
Passing to $\eta$ in $\theta_1$ integral
$$
\int  \frac{|\xi|^{s+a+\frac32} \la  \theta\ra^{1-4s+} |\xi|^{4s+4a -6 } |\xi_1|^{s+a -\tfrac32  } |\xi|^{2s} d\xi_1d\eta  }{|\xi_1|^{2s-\frac12 } |\eta|^{\frac12} \big\la \lambda +\frac{\theta^2}\xi-\xi_1^5-(\xi-\xi_1)^5+\eta \big\ra^{1-} } 
$$
$$
\les \int |\xi|^{7s+5a+\frac32-6} \la  \theta\ra^{1-4s+} |\xi_1|^{-s+a -1}  d\xi_1 \les |\xi|^{7s+5a+\frac32-6+3-12s+} \les 1,
$$
provided that $a<s+\frac3{10}$.

Region 3 is $|\eta|\gtrsim \xi^4|\xi_1|$.  Without loss of generality $|\theta-\theta_1|\gtrsim |\theta|$.  Passing to the $\eta$ variable,  we bound the integral  by 
\be\label{R131}
\int  \frac{|\xi|^{s+a+\frac32} \la \xi,\theta\ra^{1-2s+}  |\xi_1|^{\frac12}d\xi_1d\eta   }{\la\xi_1,\theta_1\ra^{2s } |\eta|^{2-s-a  } \big\la \lambda +\frac{\theta^2}\xi-\xi_1^5-(\xi-\xi_1)^5+\eta \big\ra^{1-} },
\ee
 where $\theta_1=\frac{\theta\xi_1}{\xi}\pm\sqrt{|\tfrac{\eta\xi_1(\xi-\xi_1)}\xi|}$. When  $|\theta|< |\xi|$ or $s>\frac12$, after integrating in $\eta$ we bound this by 
 $$
 \int_{1\ll|\xi_1|\ll|\xi|}  \frac{|\xi|^{s+a+\frac32} |\xi|^{1-2s+}  |\xi_1|^{\frac12-2s}d\xi_1   }{  (\xi^4|\xi_1|)^{2-s-a - }  }\les |\xi|^{3s+5a-\frac{11}2+}\les 1
 $$
 provided that $a<s+\frac12$ and $a<\frac{11}{10}-\frac{3s}5$. The second restriction is acceptable since for $\frac12<s<\frac32-\frac13$,  we have $\frac{11}{10}-\frac{3s}5>\frac13$.
 
 When $\frac14<s\leq \frac12$ and $|\theta|>|\xi|$, using
 $$
 1<|\eta\xi_1|\approx |\tfrac{\xi_1}{\xi}\theta-\theta_1|^2
 $$
 we have
 $$
 \la\xi_1,\theta_1\ra^{2s } |\eta\xi_1|^s\gtrsim \la \xi_1,\tfrac{\theta\xi_1}\xi\ra^{2s}\gtrsim |\xi_1|^{2s}|\theta|^{2s}|\xi|^{-2s},
 $$
 Using this, we bound \eqref{R131} by
 $$
 \int  \frac{|\xi|^{3s+a+\frac32}|\theta|^{1-4s+}  |\xi_1|^{\frac12-s}d\xi_1d\eta   }{  |\eta|^{2-2s-a  } \big\la \lambda +\frac{\theta^2}\xi-\xi_1^5-(\xi-\xi_1)^5+\eta \big\ra^{1-} },
 $$
 $$
 \les  \int_{1\ll|\xi_1|\ll|\xi|}   \frac{|\xi|^{3s+a+\frac32}|\xi|^{1-4s+}  |\xi_1|^{\frac12-s}d\xi_1   }{  (\xi^4|\xi_1|)^{2-2s-a-  }  }\les  \frac{|\xi|^{3s+a+\frac32}|\xi|^{1-4s+}  |\xi |^{-\frac12+s+a+}   }{  (\xi^4 )^{2-2s-a-  }  }\les 1
 $$
 provided that $a<1-\frac{4s}3$.
 
Subcase 1.ii:  $ |\xi_1|\approx |\xi-\xi_1|\gtrsim  |\xi|$.  
 
 In this case we have
$$
M\approx \xi_1^4|\xi |+|\eta|,\quad |\eta|\approx  (\tfrac{\theta-\theta_1}{\xi-\xi_1}-\tfrac{\theta_1}{\xi_1})^2\tfrac{\xi_1^2}{|\xi|}.
$$
Also assuming that $|\theta-\theta_1|\gtrsim |\theta|$ without loss of generality and then passing to the $\eta$ variable, we bound the integral by
$$
\int \frac{|\xi|^{s+a+\frac32} \la \xi,\theta\ra^{1-2s+}  d\xi_1d\theta_1   }{\la\xi_1,\theta_1\ra^{2s } |\xi\xi_1^4|^{\frac32-s-a  } \big\la \lambda +\frac{\theta^2}\xi-\xi_1^5-(\xi-\xi_1)^5+\eta \big\ra^{1-} } 
$$
\be\label{R21}
\les \int  \frac{|\xi|^{s+a+1} \la \xi,\theta\ra^{1-2s+} |\xi_1| d\xi_1d\eta   }{\la\xi_1,\theta_1\ra^{2s } |\xi\xi_1^4|^{\frac32-s-a  }
|\eta|^{\frac12} \big\la \lambda +\frac{\theta^2}\xi-\xi_1^5-(\xi-\xi_1)^5+\eta \big\ra^{1-} }
\ee
 where $\theta_1=\frac{\theta\xi_1}{\xi}\pm\sqrt{|\tfrac{\eta\xi_1(\xi-\xi_1)}\xi|}$. When  $|\theta|< |\xi|$ or $s>\frac12$, after integrating in $\eta$ we bound this by 
 $$
\int  \frac{|\xi|^{2a+\frac12}   |\xi_1|^{2s+4a-5} d\xi_1    }{ 
 \big\la \lambda +\frac{\theta^2}\xi-\xi_1^5-(\xi-\xi_1)^5  \big\ra^{\frac12-} } \les |\xi|^{2s+6a+\frac12-5 -\frac32 +}  \les 1
$$
provided that
$a<1-\frac{s}3$. In the second to last inequality we used the $\rho$ trick as in the estimate of \eqref{rhotrick} above.

When $s< \frac12 $ and $|\theta|>|\xi|$, noting that 
$$
|\eta|\frac{\xi_1^2}{|\xi|}\approx |\tfrac{\xi_1}{\xi}\theta-\theta_1|^2
$$
we obtain
$$
 \la\xi_1,\theta_1\ra^{2s } \la \eta\ra^s |\xi_1^2\xi^{-1}|^s\gtrsim \la \xi_1,\tfrac{\theta\xi_1}\xi\ra^{2s}\gtrsim |\xi_1|^{2s}|\theta|^{2s}|\xi|^{-2s}.
 $$
Using this, we bound \eqref{R21} by
$$\int  \frac{|\xi|^{s+a+1} |\theta|^{1-4s+} \la \eta\ra^s | \xi_1^2\xi^{-1}|^s|\xi|^{2s}|\xi_1| d\xi_1d\eta   }{|\xi_1|^{2s}  |\xi\xi_1^4|^{\frac32-s-a  }
|\eta|^{\frac12 } \big\la \lambda +\frac{\theta^2}\xi-\xi_1^5-(\xi-\xi_1)^5+\eta \big\ra^{1-} }
$$
$$
\les \int_{|\xi_1|\gtrsim|\xi| }  |\xi|^{3s+2 a-\frac12} |\xi|^{1-4s+} |\xi_1|^{4s+4a-5} d\xi_1   \les |\xi|^{3s+6a-\frac72}\les 1,
$$
provided that $a<\frac7{12}-\frac{s}2$.

\noindent 
\textsc{Case 2:}  $\la \tau_1\ra \gtrsim \max(M,\la\tau \ra)$. The case $\la \tau_2\ra$ is larger is identical. 

In the case $s>\frac12$, we estimate the integral in \eqref{eq:ntp} by (assuming $\la\xi-\xi_1,\theta-\theta_1\ra\gtrsim \la\xi ,\theta \ra$ without loss of generality and integrating in $\lambda_1$ and then in $\theta_1$ and $\xi_1$)
$$
\int  \frac{ |\xi|^{s+a+\frac32}\la \xi,\theta\ra^{1-2s+} d\xi_1d\theta_1}{\la\xi_1,\theta_1\ra^{2s} |\xi^3\xi_1(\xi-\xi_1)|^{1-}}\les |\xi|^{s+a+\frac32-4+} \les 1
$$
provided that $a<\frac52-s$.

Similarly, when $ \frac14<s<\frac12$ and   $|\theta|\les |\xi|$, we estimate \eqref{eq:ntp} by
 $$
\int   \frac{ |\xi|^{s+a+\frac52+}  d\xi_1d\theta_1}{\la \theta_1\ra^{2s} \la \theta-\theta_1\ra^{2s} |\xi^3\xi_1(\xi-\xi_1)|^{1-}}\les 
|\xi|^{s+a+\frac52-4+} \les 1
$$
provided that $a+s<\frac32$. 

It remains to consider the case  $ \frac14<s<\frac12$ and  $|\theta|\gg|\xi|$.   Instead of \eqref{eq:ntp}, it suffices to prove
$$
\sup_{|\xi_1|>1,\theta_1,\lambda_1} \int_{ |\theta|\gg|\xi|} \frac{|\xi|^{s+a+\frac32} \la \theta\ra^{1+} \la\tau\ra^{s+a -\tfrac32  } d\xi d\theta d\lambda}{\la\xi_1,\theta_1\ra^{2s }\la\xi-\xi_1,\theta-\theta_1\ra^{2s} \la \tau_1\ra^{1-}  \la  \tau_2\ra^{1-} } <\infty.
$$
Evaluating the $\lambda $ integral and assuming  that $|\theta-\theta_1|\gtrsim |\theta|$, we consider
$$
\int  \frac{|\xi|^{s+a+\frac32} \la  \theta\ra^{1-2s+}   d\xi d\theta}{\la\xi_1,\theta_1\ra^{2s }  \la M\ra^{1-}  \la \lambda_1+(\xi-\xi_1)^5-\xi^5+\tfrac{\theta^2}\xi-\tfrac{(\theta-\theta_1)^2}{\xi-\xi_1}\ra^{\frac32-a-s-} }.
$$
 Passing to the $\eta$ variable, we have
 $$
\int  \frac{|\xi|^{s+a+2} |\xi-\xi_1|^{\frac12}\big|\frac{\theta_1\xi}{\xi_1}\pm \sqrt{|\frac{\eta\xi(\xi-\xi_1)}{\xi_1}|}\big|^{1-2s+}   d\xi d\eta}{|\xi_1|^{\frac12} |\eta|^{\frac12} \la\xi_1,\theta_1\ra^{2s }  \la\xi^3\xi_1(\xi-\xi_1) \ra^{1-}  \la \lambda_1+\tfrac{\theta_1^2}{\xi_1}+(\xi-\xi_1)^5-\xi^5+\eta \ra^{\frac32-a-s-} }
$$
$$\les \int \frac{|\xi|^{-s+a +}|\theta_1|^{1-2s+}   d\xi d\eta}{|\xi_1|^{\frac52-2s-}  |\xi-\xi_1|^{\frac12-} |\eta|^{\frac12} \la\xi_1,\theta_1\ra^{2s }     \la \lambda_1+\tfrac{\theta_1^2}{\xi_1}+(\xi-\xi_1)^5-\xi^5+\eta \ra^{\frac32-a-s-} }
$$
$$
+\int  \frac{|\xi|^{ a-\frac12+}    d\xi d\eta}{|\xi_1|^{2-s+} |\xi-\xi_1|^{ s-} |\eta|^{s-} \la\xi_1,\theta_1\ra^{2s }   \la \lambda_1+\tfrac{\theta_1^2}{\xi_1}+(\xi-\xi_1)^5-\xi^5+\eta \ra^{\frac32-a-s-} } 
$$
Since $a,s<\frac12$, we can integrate in $\eta$ to obtain
$$
\les  \int  \frac{|\xi|^{-s+a +}   d\xi  }{|\xi_1|^{\frac52-2s-}  |\xi-\xi_1|^{\frac12-}      \la \lambda_1+\tfrac{\theta_1^2}{\xi_1}+(\xi-\xi_1)^5-\xi^5  \ra^{1-a-s-} }
$$
$$
+\int  \frac{|\xi|^{ a-\frac12+}    d\xi }{|\xi_1|^{2+s+} |\xi-\xi_1|^{ s-}      \la \lambda_1+\tfrac{\theta_1^2}{\xi_1}+(\xi-\xi_1)^5-\xi^5  \ra^{\frac12-a -} }. 
$$
 We consider two regions  $1<|\xi_1|\ll|\xi|$ and  $|\xi_1|\gtrsim |\xi|$. In the second region we estimate the integrals by $|\xi_1|^{s+a-2+}+|\xi_1|^{a-2s-\frac32+}$,
 which suffices. In the first region by letting $ \rho=\xi^5-(\xi-\xi_1)^5$, and noting that
$$|\xi|\approx |\rho|^{\frac14}|\xi_1|^{-\frac14},\quad  d\rho\approx |\xi_1||\xi|^3 d\xi\approx |\rho|^{\frac34}|\xi_1|^{\frac14} d\xi,$$
 we have
 $$
 \les \int  \frac{|\rho|^{\frac{a}4-\frac{s}4-\frac78+}   d\rho }{|\xi_1|^{\frac52-2s+\frac{-s+a-\frac12}4+\frac14+}      \la \lambda_1+\tfrac{\theta_1^2}{\xi_1} -\rho \ra^{1-a-s-} } +\int  \frac{|\rho|^{\frac{a}4-\frac{s}4-\frac78+}   d\rho }{|\xi_1|^{ 2+s+\frac{-s+a-\frac12}4+\frac14+}    \la \lambda_1+\tfrac{\theta_1^2}{\xi_1} -\rho \ra^{\frac12-a-} }
 $$
 $$
 \les \int \frac{|\rho|^{\frac{a}4-\frac{s}4-\frac78+}   d\rho }{     \la \lambda_1+\tfrac{\theta_1^2}{\xi_1} -\rho \ra^{\frac12-a-} } \les 1
 $$
 provided that $a<\frac3{10}+\frac{s}5$.
 
 \section{Local well--posedness and smoothing} \label{sec:lwp}
The assertion of Theorem~\ref{thm:local} follows from the a priori linear estimates established in Section~\ref{sec:lin} and the nonlinear estimates in Section~\ref{sec:nonlin}. The proof follows along the lines of the proof of Theorem 1.3 in Section 5 of \cite{etza} (also see the proof of Theorem 2.4 in Section 4 of \cite{etnls}). In particular, the fix point argument for equation \eqref{eq:duhamel} in the $X^{s,b}$ space
for the linear solution follows from   Proposition~\ref{w2:xsbb1}, the discussion preceeding it, and the Kato smoothing bound Proposition~\ref{freekp2:y=0}.  For the nonlinear terms we use Theorem~\ref{thm:nonlin1}, Theorem~\ref{thm:nonlin2}, and the Proposition~\ref{prop:nonlinKato}. We also use the properties \eqref{eq:xs1}, \eqref{eq:xs2}, \eqref{eq:xs3} of $X^{s,b}$ spaces. 
Similarly, the   solution   belongs to 
$  C^0_tH^s_{x,y}([0,T]\times U) \cap C^0_y\cH^s_{x,t}(\R^+ \times \R\times [0,T]) $ using  Proposition~\ref{freekp2:y=0}, Proposition~\ref{prop:verttrace}, Proposition~\ref{prop:cHstoHs}, and  Proposition~\ref{prop:nonlinKato}. These a priori estimates also imply continuous dependence  on initial data, see  Section 5 of \cite{etza}. 
Note that the solution is unique once we fix an extension of the initial data, however it is not clear whether the restriction of the solution to the half plane is independent of the extension.

Finally, the smoothing bound in Theorem~\ref{thm:smooth} follows from the same estimates as in the proof of Theorem 1.1 in  Section 5 of \cite{etnls}. 
\end{proof}
\section{Appendix}

In this appendix we present two elementary lemmas that have been used in this paper repeatedly. For the proof of the first lemma see  the Appendix of \cite{erdtzi1}. The second lemma is the well--known Schur's test.
\begin{lemma}\label{lem:sums}   If  $\beta\geq \gamma\geq 0$ and $\beta+\gamma>1$, then
\be\nn
\int_\R \frac{1}{\la \tau-k_1\ra^\beta \la \tau-k_2\ra^\gamma} d\, \tau \lesssim \la k_1-k_2\ra^{-\gamma} \phi_\beta(k_1-k_2),
\ee
where
 \be\nn
\phi_\beta(k):=\sum_{|n|\leq |k|}\frac1{\la n\ra^\beta}\sim \left\{\begin{array}{ll}
1, & \beta>1,\\
\log(1+\la k\ra), &\beta=1,\\
\la k \ra^{1-\beta}, & \beta<1.
 \end{array}\right.
\ee
The statement remains valid when $\la \tau-k_2\ra$ is replaced with $| \tau-k_2|$ provided that $\gamma<1$.
\end{lemma} 

\begin{lemma}\label{lem:schur} Let $T$ be an integral operator  with kernel $K(\theta,\eta)$, $\theta\in\R^m$, $\eta\in\R^n$. Assume that for some positive functions $p(\theta)$, $q(\eta)$, and some constants $A, B$ we have 
$$
\int |K(\theta,\eta)| p(\theta) d\theta \leq A q(\eta), \text{ for a.e. } \eta,
$$ 
$$
\int |K(\theta,\eta)| q(\eta) d\eta \leq B p(\theta), \text{ for a.e. } \theta,
$$ 
then $\|T\|_{L^2\to L^2}\leq \sqrt {AB}$.
\end{lemma}

 \end{document}